\documentclass[11pt]{amsart}

	\usepackage{amsmath,amssymb,amsthm,amsfonts,enumerate,mathrsfs,mathtools,stmaryrd,extarrows}

	\usepackage[a4paper,margin=1.2in]{geometry}  
	\usepackage[all]{xy}
 	\usepackage[utf8]{inputenc}
 	\usepackage[T1]{fontenc} 
	\usepackage[colorlinks,citecolor=magenta,linkcolor=black]{hyperref}
	
\newtheorem{thm}{Theorem}[subsection]
\newtheorem{lemma}[thm]{Lemma}
\newtheorem{prop}[thm]{Proposition}
\newtheorem{cor}[thm]{Corollary}
\newtheorem{conj}{Conjecture}[subsection]

\theoremstyle{definition} 
\newtheorem{defin}[thm]{Definition}
\newtheorem{remark}[thm]{Remark}
\newtheorem{remarks}[thm]{Remarks}
\newtheorem{example}[thm]{Example}
\newtheorem{question}[thm]{Question}

\newcommand{\bb}[1]{\mathbf{#1}} 
\newcommand{\CC}{\bb{C}}
\newcommand{\FF}{\bb{F}}

\newcommand{\QQ}{\bb{Q}}
\newcommand{\RR}{\bb{R}}

\newcommand{\ZZ}{\bb{Z}}

\newcommand{\GG}{\bb{G}}

\newcommand{\cO}{\mathscr{O}}

\renewcommand{\phi}{\varphi}

\newcommand{\ra}{\longrightarrow}
\newcommand{\xto}{\xrightarrow}
\newcommand{\xra}{\xlongrightarrow}

\newcommand{\isomto}{\xrightarrow{\,\smash{\raisebox{-0.65ex}{\ensuremath{\scriptstyle\sim}}}\,}}
\newcommand{\isomlong}{\xlongrightarrow{\,\smash{\raisebox{-0.65ex}{\ensuremath{\scriptstyle\sim}}}\,}}

\newcommand{\isom}{\simeq}
\renewcommand{\bigwedge}{\mbox{\large $\wedge$}}

\DeclareMathOperator{\Hom}{Hom}
\DeclareMathOperator{\cHom}{\mathscr{H}om}
\DeclareMathOperator{\Spec}{Spec}
\DeclareMathOperator{\Spf}{Spf}

\DeclareMathOperator{\Def}{Def}
\DeclareMathOperator{\Ext}{Ext}
\DeclareMathOperator{\Pic}{Pic}
\DeclareMathOperator{\gr}{gr}
\DeclareMathOperator{\HW}{HW}
\newcommand{\et}{{\rm\acute{e}t}}

\newcommand{\wt}[1]{{\mathchoice%
  {\raisebox{1.5ex}{\resizebox{1.7ex}{!}{{}\hphantom{i}\ensuremath{{\sim}}}} \hspace{-1.7ex}{#1}}%
  {\smash{\raisebox{1.5ex}{\resizebox{1.7ex}{!}{{}\hphantom{i}\ensuremath{{\sim}}}}\hspace{-1.7ex}{#1 }}\vphantom{\tilde I}}%
  {\raisebox{1.1ex}{\resizebox{1.3ex}{!}{{}\hphantom{i}\ensuremath{{\sim}}}}\hspace{-1.3ex}{#1}}%
  {\raisebox{0.8ex}{\resizebox{1ex}{!}{{}\hphantom{i}\ensuremath{{\sim}}}}\hspace{-1ex}{#1}}%
}}

\newcommand{\swt}[1]{{\mathchoice%
  {\raisebox{0.9ex}{\resizebox{1.2ex}{!}{\ensuremath{{\sim}}}}\hspace{-1.4ex}{#1}}%
  {\smash{\raisebox{0.9ex}{\resizebox{1.2ex}{!}{\ensuremath{{\sim}}}}\hspace{-1.4ex}{#1 }}\vphantom{I}}%
  {\raisebox{0.7ex}{\resizebox{0.8ex}{!}{\ensuremath{{\sim}}}}\hspace{-0.9ex}{#1}}%
  {\raisebox{0.5ex}{\resizebox{1ex}{!}{{}\hphantom{i}\ensuremath{{\sim}}}}\hspace{-1ex}{#1}}%
}}

	\numberwithin{equation}{subsection}

 	\author[P. Achinger]{Piotr Achinger}
 	\address{Institute of Mathematics, Polish Academy of Sciences
    	\newline\indent ul. Śniadeckich 8, 00-656 Warsaw, Poland}
 	\email{pachinger@impan.pl}

 	\author[M. Zdanowicz]{Maciej Zdanowicz}
 	\address{\'Ecole Polytechnique F\'ed\'erale de Lausanne, Chair of Algebraic Geometry \newline 
    \indent MA C3 585 (Bâtiment MA), Station 8, CH-1015 Lausanne}
 	\email{maciej.zdanowicz@epfl.ch}

 	\title{Serre--Tate theory for Calabi--Yau varieties}
 	\date{\today}

\begin{document}

\begin{abstract}
	Classical Serre–Tate theory describes deformations of ordinary abelian varieties. It implies that every such variety has a canonical lift to characteristic zero and equips its local moduli space with a Frobenius lifting and canonical multiplicative coordinates. A variant of this theory has been obtained for ordinary K3 surfaces by Nygaard and Ogus.

	In this paper, we construct canonical liftings modulo $p^2$ of varieties with trivial canonical class which are ordinary in the weak sense that the Frobenius acts bijectively on the top cohomology of the structure sheaf. Consequently, we obtain a Frobenius lifting on the moduli space of such varieties. The quite explicit construction uses Frobenius splittings and a relative version of Witt vectors of length two. If the variety has a smooth deformation space and bijective first higher Hasse--Witt operation, the Frobenius lifting gives rise to canonical coordinates. One of the key features of our liftings is that the crystalline Frobenius preserves the Hodge filtration.

	We also extend Nygaard's approach from K3 surfaces to higher dimensions, and show that no nontrivial families of such varieties exist over simply connected bases with no global one-forms.
\end{abstract}

\maketitle


\section{Introduction}
\label{s:intro}


\subsection{Deformations of varieties with trivial canonical class} 

Let $X$ be a smooth and projective algebraic variety with trivial canonical class (i.e.\ $c_1(X) = 0 \in \Pic X$) over an algebraically closed field $k$. 
A fundamental result, due to Bogomolov \cite{Bogomolov}, Tian \cite{Tian}, and Todorov \cite{Todorov} (see also \cite{Kawamata,KawamataErratum,Ran}), and called the BTT theorem, states that if $k=\CC$, then deformations of $X$ are unobstructed; in other words, its local deformation space -- or Kuranishi space -- is smooth. Moreover, the base of a~modular family of such $X$ carries a structure of an affine manifold \cite{LiuShen} (equivalently, a torsion-free flat holomorphic connection on the tangent bundle), giving rise to local canonical flat coordinates.

In contrast, if $k$ has characteristic $p>0$, then the assertion of the BTT theorem is no longer true (see e.g.\ \cite{Hirokado,SchroerExamples}), even if one considers deformations in the equi-characteristic direction \cite[Proposition~5.4]{CynkVanStraten}. Important results towards a characteristic $p$ version of the BTT theorem were given by Schr\"oer \cite{SchroerT1} and Ekedahl and Shepherd-Barron \cite{EkedahlSB}, but in general the deformation theory of varieties with trivial canonical class in positive characteristic remains a~mystery (for example, both aforementioned results require liftability to characteristic zero to prove unobstructedness). 

On the other hand, all known examples of varieties with trivial canonical class in positive characteristic which do not lift to characteristic zero are `supersingular' in the broad sense that the Hasse--Witt operation, i.e.\ the map
\begin{equation} \label{eqn:hasse-witt}
	F^* \colon H^d(X, \cO_X)\ra H^d(X, \cO_X), \quad d=\dim X
\end{equation}
induced by the absolute Frobenius, vanishes. It is a natural expectation (which we have not seen stated explicitly in the literature) that this is the case in general. More precisely, we conjecture the following.

\begin{conj} \label{conj:folk}
	Let $X$ be a smooth projective variety with trivial canonical class over an algebraically closed field of characteristic $p>0$. Suppose that the Hasse--Witt operation \eqref{eqn:hasse-witt} is bijective. Then $X$ has unobstructed deformations over the ring of Witt vectors $W(k)$ of~$k$.
\end{conj}

The goal of this paper is to investigate the deformation theory of varieties with trivial canonical class for which \eqref{eqn:hasse-witt} is bijective, with a special focus on deformations modulo~$p^2$.


\subsection{Classical Serre--Tate theory} 

In the important case of ordinary abelian varieties, a complete description of the deformation problem was given by Serre and Tate \cite{KatzSerreTate,DeligneIllusieKatz}. Let us summarize the main points of this description: let $A$ be an ordinary abelian variety (or, somewhat more generally, an ordinary variety with trivial tangent bundle \cite{MehtaSrinivas}) over a~perfect field $k$ of characteristic $p$, and let $\Def_A$ be its formal deformation space over $W(k)$. Then:
\begin{itemize}
	\item $\Def_A$ has the structure of a \emph{formal group}, in fact a formal torus, over $W(k)$,
	\item $\Def_A$ has a canonical \emph{lifting of Frobenius} $\wt F$  (the $p$-th power map in the group law),
	\item $\Def_{A\otimes \bar k}$ has canonical \emph{multiplicative coordinates} (dependent on the choice of basis of the $p$-adic Tate module $T_p(A(\bar k))$), i.e.\ there is a~preferred isomorphism
	\[ 
		\Def_{A\otimes \bar k} \isom \Spf W(k)[[q_{ij} - 1, \,\, 1\leq i,j  \leq \dim A]].
	\]
	They are compatible with the Frobenius lifting $\wt F$ in the sense that $\wt F{}^*(q_{ij}) = q_{ij}^p$.
	\item Setting $q_{ij}=1$, we obtain a \emph{canonical lifting} $\wt A$ over $W(k)$, which can be characterized as the unique lifting to which the Frobenius $F_A$ lifts.
	\item The restriction map $\Pic \wt A\to \Pic A$ admits a natural section. Consequently, the canonical lift $\wt A$ is projective, and hence algebraizable.
\end{itemize}
From our viewpoint, it is the Frobenius lifting $\wt F$ which is the  most fundamental; in fact, it determines the other features uniquely (cf.\ \cite[Appendix]{DeligneIllusieKatz}). It can be regarded as an analog of the affine manifold structure on the moduli of complex tori (see \cite{Mochizuki}).

As an important step in his proof of the Tate conjecture for ordinary K3 surfaces \cite{Nygaard}, Nygaard has developed an appropriate version of Serre--Tate theory. 
His results were recently extended to Calabi--Yau threefolds by Ward in his thesis \cite{WardThesis}. In fact, the same methods can be used to yield an analog for higher-dimensional varieties, under some extra assumptions. 

Before stating our first result, we have to note that beyond the cases of abelian varieties and K3 surfaces, it is important to distinguish between several notions of ordinarity. In this paper, the following conditions on a smooth projective variety $X$ of dimension $d$ shall play a role:

\begin{defin}
Let $X$ be a smooth projective variety with trivial canonical class over an algebraically closed field $k$ of characteristic $p>0$.
\begin{enumerate}[(a)]
	\item $X$ is \emph{$1$-ordinary} if the Hasse--Witt operation \eqref{eqn:hasse-witt} is bijective,
	\item $X$ is \emph{$2$-ordinary} if it is $1$-ordinary and if the first higher Hasse--Witt operation
	\[
		\HW(1)\colon H^{d-1}(X, \Omega^1_X)\ra H^{d-1}(X, \Omega^1_X)
	\]
	is bijective (see Definition~\ref{def:m-ord}).
	\item $X$ is \emph{ordinary} (in the sense of Bloch--Kato) if $H^j(X, d\Omega^i_X)=0$ for all $i,j$.
\end{enumerate}
\end{defin}

\noindent 
If $X$ has no crystalline torsion, ordinarity is equivalent to the equality of the Hodge and Newton polygons for $H^d_{\rm cris}(X/W(k))$  \cite[\S 7]{Bloch_Kato}, and implies $2$-ordinarity. More generally, $m$-ordinarity means that the first $m$ segments of the Hodge and Newton polygon coincide. Note that a choice of trivialization of the canonical bundle $\omega_X$ yields an identification of $H^{d-1}(X, \Omega^1_X)$ with the dual of the tangent space $H^1(X, T_X)$ of the deformation functor $\Def_X$.

Our first result is the following.

\begin{thm}[Corollary~\ref{cor:nygaard}] \label{thm:nygaard}
	Let $X$ be a smooth projective variety with trivial canonical class over an algebraically closed field of characteristic $p>0$. Assume that
	\begin{enumerate}[(i)] 
		\item $H^{1}(X, \cO_X) = 0 = H^0(X, T_X)$, 
		\item $X$ is $2$-ordinary,
		\item the crystalline cohomology groups $H^*_{\rm cris}(X/W(k))$ are torsion-free for $*=d,d+1$,
		\item the Hodge spectral sequence of $X/k$ degenerates,
		\item $X$ has unobstructed deformations over $W_m(k)$ for some $m\geq 1$.
	\end{enumerate}
	Then a choice of a basis of the free $\ZZ_p$-module 
	\[ 
		\Hom_{\ZZ_p}(H^d_{\rm cris}(X/W(k))^{\phi = p}, H^d_{\rm cris}(X/W(k))^{\phi = 1})
	\] 
	gives rise to an isomorphism
	\[ 
		\Def_{X/W_m(k)} \isom \Spf W_m(k)[[q_1-1, \ldots, q_r-1]], \quad r = \dim H^{d-1}(X, \Omega^1_X),
	\]
	and a natural lifting of Frobenius, defined by $\wt F{}^*(q_i)=q_i^p$.
\end{thm}


\subsection{\texorpdfstring{$F$-splittings and canonical liftings modulo $p^2$}{F-splittings and canonical liftings modulo p2}}

Theorem~\ref{thm:nygaard} is unsatisfying in several aspects. First, the assumption on smoothness of the deformation space is discouraging. Second, its relation to Hodge theory (Frobenius and the Hodge filtration) is a mystery (see Question~\ref{question:crysf1}). It is also important to ask whether the $2$-ordinarity assumption could be relaxed to $1$-ordinarity. Our main new insight is that using a different approach one can build a version of the Serre--Tate theory for $1$-ordinary varieties with trivial canonical class which describes deformations modulo $p^2$ (i.e.\ over Artinian $W_2(k)$-algebras).



An important feature of our approach is the use of $F$-splittings, which are $\cO_X$-linear splittings of the Frobenius map $F^*\colon \cO_X\to F_* \cO_X$. They were invented by Mehta and Ramanathan in the 1980s \cite{MehtaRamanathan} and were put to good use in geometric representation theory in positive characteristic, but so far they have not appeared too prominently in arithmetic geometry (see, however, \cite{MehtaSrinivas}). This is in contrast with the use of formal and $p$-divisible groups in the classical case, and it would be interesting to see a more direct link between $F$-splittings and the formal groups of Artin and Mazur. 

Our key starting point is the following construction. Let $X$ be a $1$-ordinary variety with trivial canonical class. Then $X$ admits a unique $F$-splitting $\sigma$. On the other hand, a construction due to the second author \cite{Zdanowicz} attaches to a pair $(X, \sigma)$ of a $k$-scheme and an $F$-splitting a canonical lifting $\wt X$ of the Frobenius twist $X'$ of $X$. The structure sheaf of $\wt X$ can be described as the following quotient of $W_2(\cO_X)$:
\[
	\cO_{\wt X} = W_2 (\cO_X) /I_\sigma, 
	\quad
	I_\sigma = \{(0, f)\,|\, \sigma(f)=0\}.
\]
It was known before \cite[8.5]{IllusieFrobenius} that an $F$-split scheme can be lifted modulo $p^2$. 
What is new here is that there is a~preferred, and quite explicit, such lifting. This was a strong indication that there should be a version of Serre--Tate theory for $1$-ordinary varieties with trivial canonical class.

Already the existence of the preferred $\wt X$ has interesting consequences. First, by \cite{DeligneIllusie}, for every $n<p$ there exists a canonical Hodge decomposition
\[ 
	H^n_{\rm dR}(X/k) \isom \bigoplus_{i+j=n} H^j(X', \Omega^i_{X'/k}).
\]
One wonders what other features of classical Hodge theory have analogues for $1$-ordinary varieties with trivial canonical class. Second, by construction the canonical lifting $\wt X$ is a closed subscheme of $W_2(X)$, and the Teichm\"uller lifting
\[ 
	\cO^\times_X\ra \cO_{\wt X}^\times, \quad f \mapsto \text{the class of }(f, 0)\text{ mod }I_\sigma
\]
gives rise to a natural splitting of the restriction map $\Pic \wt X\to \Pic X$. 

From our perspective, the most important feature of the canonical lifting $\wt X$ is seen through its de Rham cohomology. Recall that the isomorphism
\[ 
	\rho \colon H^d_{\rm cris}(X'/W_2(k)) \isomlong H^d_{\rm dR}(\wt X/W_2(k)
\]
endows the right hand side with the crystalline Frobenius $\phi$.

\begin{thm}[{Theorem~\ref{thm:fhodge1}}] \label{thm:preserve-f1} 
	Let $X$ be a $1$-ordinary smooth projective variety with trivial canonical class, and let $\wt X$ be the canonical lifting of $X'$ described above. Then the crystalline Frobenius
	\[ 
		\phi \colon H^d_{\rm dR}(\wt X/W_2(k))\ra H^d_{\rm dR}(\wt X/W_2(k))
	\]
	preserves the submodule $F^1 H^d_{\rm dR}(\wt X/W_2(k))$ (the image of $H^d(\wt X, \Omega^{\bullet \geq 1}_{\wt X}) \to H^d_{\rm dR}(\wt X/W_2(k))$).
\end{thm}

\noindent
If $p>2$ and certain technical assumptions are satisfied, $\wt X$ is the unique lifting $\wt X$ of $X'$ with this property (Theorem~\ref{thm:canlift-uniq}). By results of Katz \cite[Appendix]{DeligneIllusieKatz}, this implies that for abelian varieties and K3 surfaces, the lifting $\wt X$ agrees with the usual one, obtained by Serre--Tate theory (Corollary~\ref{cor:comparison-st1}).
 

\subsection{Modular Frobenius liftings}

In order to describe the deformation theory of a $1$-ordinary variety $X$ with trivial canonical class, we construct the canonical lifting in families: if $X/S$ is a family of such varieties over a $k$-scheme $S$, and if $\wt S$ is a flat lifting of $S$ over $W_2(k)$, there is a canonical lifting $\wt{X}/ \wt S$ of the Frobenius pull-back family $X'=F_S^* X$. This is achieved using a construction of a sheaf of rings $W_2(\cO_{X}/\wt S)$ on $X$, a relative version of $W_2\cO_X$. As in the absolute case, the unique relative $F$-splitting $\sigma$ of $X/S$ gives rise to an ideal $I_\sigma \subseteq W_2(\cO_{X}/\wt S)$, and $\cO_{\wt{X}} = W_2(\cO_{X}/\wt S)/I_\sigma$.

Note that in the relative case (where $S$ is not perfect), it is important to have in mind that it is $X'$, and not $X$, that admits a canonical lifting. For example, if $S$ is smooth over $k$, then $X'/S$ will always have vanishing Kodaira--Spencer map. Conversely, one can show that, if the fibers have no nonzero global vector fields, a family with vanishing Kodaira--Spencer descends canonically under Frobenius (see Appendix~\ref{app:zeroks}). Consequently, a family of $1$-ordinary varieties with trivial canonical class, no global vector fields, and vanishing Kodaira--Spencer admits a canonical lifting to any lifting of the base modulo $p^2$.

Consider now the deformation functor of $X$
\[ 
	\Def_X \colon {\rm Art}_{W_2(k)}(k) \ra {\rm Sets}.
\]
Suppose that $\Def_X \isom \wt S = \Spf R$ for a \emph{flat} $W_2(k)$-algebra $R$. Then the construction of the relative canonical lifting yields \emph{by abstract nonsense} a lifting of Frobenius
\[ 
	F_{\wt S} \colon \wt S\ra \wt S.
\]
We have thus obtained the key ingredient of a Serre--Tate theory for $X$. As in the absolute case, the key feature concerns the relationship of this Frobenius lifting with de Rham cohomology.

\begin{thm}[Theorem~\ref{thm:fhodge2}] \label{thm:intro-fhodge2}
	Let $X_0$ be a $1$-ordinary smooth projective variety with trivial canonical class, with $d=\dim X_0$. Suppose that the deformation space $\wt S = \Def_{X_0/W_2(k)}$ is pro-representable and smooth.
	Then the construction outlined above endows $\wt S$ with a~Frobenius lifting $F_{\wt S}$. Let $\wt X/\wt S$ be the universal family, and let 
	\[ 
		H = H^d_{\rm cris}(\wt X/\wt S) \isom H^d_{\rm dR}(\wt X/\wt S)
	\]
	be the associated Hodge $F$-crystal. Then the induced crystalline Frobenius map
	\[ 
		\phi(F_{\wt S})F_{\wt S}^* \colon H\ra H
	\]
	preserves the submodule $F^1 H$. 
\end{thm}

\noindent
If $p>2$ and if certain technical assumptions are satisfied, then $F_{\wt S}$ is unique with this property (Corollary~\ref{cor:uniq-froblift-moduler}). 
Again, using Katz's results, we can show that for abelian varieties and K3 surfaces, this lifting of Frobenius agrees with the one obtained by classical Serre--Tate theory (Corollary~\ref{cor:comparison-st2}). 


\subsection{\texorpdfstring{$2$-ordinarity and canonical coordinates}{2-ordinarity and canonical coordinates}}

The modular Frobenius lifting described above carries important arithmetic information related to the first higher Hasse--Witt operation and which is especially interesting if $X$ is $1$-ordinary but not $2$-ordinary. 

Recall that if $S$ is a smooth scheme over $k$, and if $\wt F\colon \wt S\to \wt S$ is a lifting of $S$ over $W_2(k)$ together with Frobenius, one has the induced operator
\begin{equation} \label{eqn:def-xi} 
	\xi\colon F^*\Omega^1_{S} \ra \Omega^1_{S}, \quad \xi(\omega) = \frac{1}{p}\wt F{}^*(\swt \omega)
\end{equation}
(here $\swt\omega\in \smash{\Omega^1_{\wt S}}$ is a lifting of the local section $\omega\in \Omega^1_{S}$). Following Mochizuki \cite{Mochizuki}, we call $\wt F$ \emph{ordinary} if $\xi$ is an isomorphism.

Applying this to $\wt S= \Def_X$ being the base of a local modular family, and specializing the operator $\xi$ at the closed point, we obtain a map
\begin{equation} \label{eqn:xi0}
	\xi(0)\colon H^1(X, T_X)^\vee\ra H^1(X, T_X)^\vee 
\end{equation}
For the statements below, it is convenient to choose a basis element of $H^d(X, \cO_X)$ which is invariant under the Hasse--Witt operation. This together with Serre duality allows us to interpret $\xi(0)$ as a map
\begin{equation} \label{eqn:xi1}
	\xi(0)\colon H^{d-1}(X, \Omega^1_X)\ra H^{d-1}(X, \Omega^1_X).
\end{equation}
\begin{thm}[{Proposition~\ref{prop:beta-hw1} + Corollary~\ref{cor:xi-hw1}}] \label{thm:alphabetahw1} 
	Suppose that $p>2$. Let $X/k$ be a $1$-ordinary variety with trivial canonical class as in Theorem~\ref{thm:intro-fhodge2}, and let $\wt X/W_2(k)$ be its canonical lifting.  Suppose that the Hodge groups of the universal deformation over $W_2(k)$ are free and that its Hodge spectral sequence degenerates. Then the following three Frobenius-linear maps are equal:
	\begin{enumerate}[(i)]
		\item The map $\xi(0)\colon H^{d-1}(X, \Omega^1_X)\to H^{d-1}(X, \Omega^1_X)$ \eqref{eqn:xi1}.
		\item The map $\beta\colon H^{d-1}(X, \Omega^1_X)\to H^{d-1}(X, \Omega^1_X)$ obtained by dividing the crystalline Frobenius on $F^1 H^d_{\rm dR}(\wt X/W_2(k))$ by $p$ (this makes sense thanks to Theorem~\ref{thm:preserve-f1}).
		\item The first higher Hasse--Witt operation $\HW(1)\colon H^{d-1}(X, \Omega^1_X)\to H^{d-1}(X, \Omega^1_X)$.
	\end{enumerate}
	Consequently, $\wt F$ is ordinary if and only if $X$ is $2$-ordinary.
\end{thm}

\noindent 
A simple corollary of this is that every $1$-ordinary variety as above can be (formally) deformed to a $2$-ordinary one (Corollary~\ref{cor:2-ordinary-def}).

The second big feature of the modular Frobenius lifting $\wt F$ is that if $X$ is $2$-ordinary, so that $\wt F$ is ordinary, we obtain an $\FF_p$-local system
\[ 
	\mathcal{M} = \{ x\in \Omega^1_S \,|\, \xi(x) = x\}
\]
such that $\Omega^1_S \isom \mathcal{M}\otimes_{\FF_p} \cO_S$. Consequently, $\Omega^1_S$ becomes trivialized on a canonical finite \'etale cover of $S$. This is the analog of the affine manifold structure over $\CC$. 

\begin{thm}[Proposition~\ref{prop:cancoord}]
	Suppose that $p>2$. Let $X$ be a $2$-ordinary variety with trivial canonical class satisfying the assumptions of Theorem~\ref{thm:alphabetahw1}. Then for every choice of a basis $\{\omega_i\}$ of the $\FF_p$-vector space 
	\[ 
		M = \ker\left( \HW(1) - {\rm id} \colon H^{d-1}(X, \Omega^1_X) \ra H^{d-1}(X, \Omega^1_X)\right)
	\]
	there is an isomorphism
	\[ 
		\Def_X = \Spf W_2(k)[[q_1-1, \ldots, q_r-1]], \quad r = \dim_k H^{d-1}(X, \Omega^1_X) = \dim_{\FF_p} M
	\]
	such that $\wt F(q_i) = q_i^p$ and the image of $d\log q_i$ in $T^* \Def_X = H^{d-1}(X, \Omega^1_X)$ equals $\omega_i$. 
\end{thm}

\noindent
These coordinates are only unique up to jet order depending on $p$: if $q'_i$ are another such coordinates, then 
\[ 
	q'_i - q_i \in (q_1^p - 1, \ldots, q_r^p - 1),
\]
where the ideal on the right is independent of the choices.


\subsection{Application to isotriviality questions}
To illustrate the utility of the modular Frobenius lifting, we give an application to isotriviality of families of $2$-ordinary varieties. Recall that by a result of Raynaud \cite[Chapter~XI, Theorem~5.1]{MoretBailly}, every family of ordinary abelian varieties over a complete curve is isotrivial. The proof of this result makes use of the global geometry of the moduli space, notably the fact that the locus of non-ordinary abelian varieties is an ample divisor. For general varieties with trivial canonical class, we do not know of such results. Nevertheless, the Frobenius lifting can be used to show that no nontrivial families of $2$-ordinary varieties exist over simply connected varieties with no nonzero global $1$-forms.

\begin{thm}[see Proposition~\ref{prop:isotriv1} and Proposition~\ref{prop:isotriv2} for precise statements]
	Let $S/k$ be a smooth simply connected variety with $H^0(S, \Omega^1_S)=0$, and let $X/S$ be a~smooth projective morphism whose fibers are $2$-ordinary varieties with trivial canonical class satifying certain assumptions. Then $f$ is a constant family.
\end{thm}


\subsection{Future directions}

\emph{1.} We were unable to construct in an analogous way a canonical lifting modulo $p^n$ for $n>2$. On the other hand, we do not know of an example of an $F$-split variety which does not admit a lifting over $W(k)$. By the results of \cite{EkedahlSB}, such an extension would give an affirmative answer to Conjecture~\ref{conj:folk}. A good first step in this direction would be to check whether in the simplest case of an ordinary elliptic curve $E$, the Serre--Tate lifting $E_n$ admits a closed immersion into $W_n(E)$ for $n>2$.

\emph{2.} The sheaves of relative Witt vectors $W_2(\cO_X/\wt S)$ seem to be new and potentially of independent interest.
It would be desirable to construct relative Witt vectors of length $n>2$, and for singular schemes. Being unable to develop such a general theory, we decided to stick to the simplest case in this paper. 

\emph{3.} The construction of a canonical lifting modulo $p^2$ for a $1$-ordinary variety with trivial canonical class can be extended to varieties with finite height. We discuss this construction, based on the ideas of Yobuko \cite{Yobuko}, in Appendix~\ref{app:yobuko}. We do not know how much of the theory developed in this paper can be extended to the case of finite height. For K3 surfaces, such a theory was developed by Nygaard and Ogus \cite{NygaardOgus}.

\emph{4.} Our construction of the canonical lifting is reasonably explicit, and it seems that in practice one could write it down with equations in the case of hypersurfaces or complete intersections. This could be interesting already for elliptic curves, where the results should interact with those of Finotti \cite{Finotti}.

\emph{5.} As we have already mentioned, for $p>d$ the canonical lifting gives by \cite{DeligneIllusie} a~canonical Hodge decomposition of $H^d_{\rm dR}(X/k)$. How can we characterize this decomposition?

\emph{6.} We do not know whether one should expect that exists a formal lifting of a (suitably) ordinary variety with trivial canonical class over $W(k)$ for which the crystalline Frobenius preserves the first step of the Hodge filtration as in Theorem~\ref{thm:preserve-f1} (we know that there is at most one), or, when such a lift does exist, whether the crystalline Frobenius should preserve the entire Hodge filtration. The latter feature could be used to produce CM Hodge structures as in the case of abelian varieties and K3 surfaces. See \S\ref{ss:fhodge1-uniq} and \S\ref{ss:nygaardp2} for more discussion.

\emph{7.} Lastly, we do not know if there is a logarithmic variant of our construction, which would allow one to extend the Frobenius lifting to some part of the boundary of the moduli space, and to study some arithmetic aspects of the toric degenerations studied by Gross and Siebert \cite{GrossSiebert}.




\label{ss:outline}

\subsection{Notation and conventions}
\label{ss:notation-conventions}

We fix a perfect field $k$ of characteristic $p>0$. We tend to denote schemes over $W_2(k)$ by $\wt X, \wt S, \ldots$ and by $X, S, \ldots$ their reductions modulo $p$. If $X/S$ is a morphism of $k$-schemes, we denote by $X'$ the pullback of $X$ along the absolute Frobenius $F_S$ of $S$, and by $F_{X/S}\colon X\to X'$ the relative Frobenius.

\subsection*{Acknowledgements}

We would like to thank Bhargav Bhatt, J\k{e}drzej Garnek, Adrian Langer, Daniel Litt, Arthur Ogus, and Lenny Taelman for helpful discussions.

Part of this work was conducted during the first author's stay at the MPIM and the Hausdorff Center for Mathematics in Bonn. 

The first author was supported by NCN SONATA grant number $2017/26/D/ST1/00913$. The second author's work was supported by Zsolt Patakfalvi's Swiss National Science Foundation Grant No. $200021 / 169639$.  This work was partially supported by the grant 346300 for IMPAN from the Simons Foundation and the matching 2015-2019 Polish MNiSW fund.

\setcounter{tocdepth}{1}
\tableofcontents


\section{Ordinary varieties with trivial canonical class}
\label{s:ord-cy}

\noindent
Let $k$ be a perfect field of characteristic $p>0$. We call a scheme $X/k$ a \emph{variety with trivial canonical class} if $X$ is smooth, projective, geometrically connected, and if the canonical bundle $\omega_X = \det \Omega^1_X$ is trivial. We do not implicitly fix however a particular trivialization of $\omega_X$, unless stated otherwise. 

\subsection{Crystalline and de Rham cohomology}

The most interesting cohomological invariant of a variety with trivial canonical class $X$ is its middle crystalline cohomology group 
\[ 
	H^d_{\rm cris}(X/W(k)), \quad d=\dim X
\]
as well as the corresponding de Rham cohomology group $H^d_{dR}(X/k)$. These are related by the short exact sequence
\[ 
	0\to H^d_{\rm cris}(X/W(k))\otimes_{W(k)} k\to H^d_{dR}(X/k) \to {\rm Tor}_1(H^{d+1}_{\rm cris}(X/W(k)), k)\to 0.
\]

\medskip

\noindent {\bf Assumption (NCT).} \emph{
	We suppose that the crystalline cohomology groups $H^*_{\rm cris}(X/W(k))$ are free $W(k)$-modules for $*=d,d+1$.}

\medskip


The de Rham cohomology groups $H^*_{\rm dR}(X/k)$ are by definition the hypercohomology groups of the de Rham complex $\Omega^\bullet_{X/k}$, and as such they are abutments of the two spectral sequences: the Hodge spectral sequence
\begin{equation} \label{eqn:hodgess} 
	E^{ij}_1 = H^j(X, \Omega^i_{X/k}) \quad \Rightarrow \quad H^{i+j}_{\rm dR}(X/k)
\end{equation}
and the conjugate spectral sequence \cite[\S 2]{KatzAlgSoln}
\begin{equation} \label{eqn:conjss}
	E^{ij}_2 = H^i(X, \mathscr{H}^j(\Omega^\bullet_{X/k})) \quad \Rightarrow \quad H^{i+j}_{\rm dR}(X/k).
\end{equation}
Further, the Cartier isomorphism $C\colon \mathscr{H}^j(F_{X/k,*}\Omega^\bullet_{X/k}) \isom \Omega^j_{X'/k}$ allows one to identify the $E^{ij}_2$ term of the conjugate spectral sequence with $H^i(X', \Omega^j_{X'/k})$. The abutment filtration on $H^*_{\rm dR}(X/k)$ induced by the Hodge resp.\ conjugate spectral sequence is called the Hodge resp.\ conjugate filtration and denoted $F^i$ resp.\ $F_i$. Explicitly,
\[ 
	F^i = {\rm Im}\left(H^*(X, \Omega^{\bullet\geq i}_X)\ra H^*_{\rm dR}(X/k)\right),
	\quad
	F_i = {\rm Im}\left(H^*(X, \tau_{\leq i} \Omega^\bullet_X)\ra H^*_{\rm dR}(X/k)\right).
\]
The first spectral sequence degenerates if and only if the second one does, and in this case we have
\begin{equation} \label{eqn:hodge-conj-gr}
	F^i/F^{i+1} \isomlong H^{*-i}(X, \Omega^i_X) \stackrel{C^{-1}}\isomlong F_i/F_{i-1}.
\end{equation}

\noindent {\bf Assumption (DEG).} \emph{
	We assume that the spectral sequences \eqref{eqn:hodgess} and \eqref{eqn:conjss} degenerate.}

\medskip
	
\noindent
By the results of \cite{DeligneIllusie}, this is the case if $p>\dim X$ and $X$ can be lifted to $W_2(k)$, in particular if $p>\dim X$ and $X$ is $1$-ordinary (see Corollary~\ref{cor:can-lift-cy}).

\subsection{Hodge and Newton polygons}\label{ss:hodge_newton_polygons}

The crystalline cohomology groups are endowed with the crystalline Frobenius
\[ 
	\phi\colon H^d_{\rm cris}(X/W(k)) \ra H^d_{\rm cris}(X/W(k)),
\]
which is semi-linear with respect to the Frobenius on $W(k)$ and which is injective modulo torsion. Its main invariants are the \emph{Hodge numbers} $h^0, \ldots, h^d$, determined by
\[ 
	H = \bigoplus H_i, \quad \phi(H) = \bigoplus p^i H_i, \quad {\rm rank}\, H_i = h^i,
\]
where $H = H^d_{\rm cris}(X/W(k))/\mathrm{torsion}$, and the \emph{Newton slopes} $\lambda_1\leq \ldots\leq \lambda_m$, $m = {\rm rank}\, H$, coming from the Dieudonn\'e--Manin decomposition of the associated $F$-isocrystal over the algebraic closure of $k$. They are conveniently encoded by means of the \emph{Hodge polygon}, which is the graph of the piecewise linear function $h\colon [0,m]\to \RR$, $h(0) = 0$, with slope $i$ on $[h^0+\ldots+ h^{i-1}, h^0+\ldots + h^i]$, and the \emph{Newton polygon}, which is the graph of the piecewise linear function $\lambda \colon [0,m]\to \RR$, $\lambda(0) = 0$, with slope $\lambda_i$ on $[i-1,i]$.

The Mazur--Ogus theorem \cite[8.39]{BerthelotOgus} asserts that under our assumptions $h^i$ equals $\dim_k H^{d-i}(X, \Omega^i_X)$, that the Newton polygon lies above the Hodge polygon, and that the endpoints of the two polygons coincide. Moreover, the image of $M_i = \phi^{-1}(p^i H)$ in $H\otimes k = H^*_{\rm dR}(X/k)$ equals $F^i$, the image of $(p^{-i} \phi)(M_i)$ in $H^*_{\rm dR}(X/k)$ equals $F_i$, and the isomorphism \eqref{eqn:hodge-conj-gr} is the one induced by $p^{-i}\phi$ [\emph{op.cit.}, 8.26].

\medskip

\noindent \emph{In the rest of this section, with the exception of \S\ref{ss:ord-family}, we assume that $X$ is a variety with trivial canonical class of dimension $d$ satisfying (NCT) and (DEG).}

\subsection{The Hasse--Witt operation}\label{ss:hasse_witt} The \emph{Hasse--Witt operations} are the maps 
\[
	\HW(0) = F_X^* \colon H^*(X, \cO_X) \ra H^*(X, \cO_X).
\]

\begin{prop} \label{prop:1-ord}
	Let $X$ be a normal proper scheme over $k$. Consider the following conditions:
	\begin{enumerate}[(i)]
		\item $X$ is $F$-split, i.e.\ the exact sequence
			\begin{equation} \label{eqn:fsplitseq} 
				0\ra \cO_{X'} \ra F_{X/k,*} \cO_X \ra B\Omega^1_X \ra 0
			\end{equation}
			is split.
		\item $H^i(X, B\Omega^1_X) = 0$ for all $i\geq 0$.
		\item The Hasse--Witt operations $\HW(0)$ on $H^i(X, \cO_X)$ are bijective for all $i$.
		\item The Hasse--Witt operation $\HW(0)$ on $H^d(X, \cO_X)$ is bijective, where $d=\dim X$.
	\end{enumerate}
	Then (i)$\Rightarrow$(ii)$\Rightarrow$(iii)$\Rightarrow$(iv). If $X$ is a variety with trivial canonical class, then all four conditions are equivalent, and a splitting of \eqref{eqn:fsplitseq} is unique if it exists.
\end{prop}

\begin{proof}
The implications (i)$\Rightarrow$(ii)$\Rightarrow$(iii) follow by taking the cohomology exact sequence of \eqref{eqn:fsplitseq}. For (iv)$\Rightarrow$(i), see \cite[Proposition~9]{MehtaRamanathan}.  
For the uniqueness, note first that splittings of \eqref{eqn:fsplitseq} are a torsor under $\Hom(B\Omega^1_X, \cO_{X'}) \isom H^0(X', B\Omega^d_X)$ \cite[\S 3]{GeerKatsura}. We have an exact sequence
\[ 
	0\ra H^0(X', B\Omega^d_X) \ra H^0(X', F_* \Omega^d_X) \xra{C} H^0(X', \Omega^d_{X'})
\]
where the map $C$ is dual to $F\colon H^d(X', \cO_{X'})\to H^d(X, \cO_X)$ and hence bijective.
\end{proof}

\begin{defin}
	We call a variety with trivial canonical class $X$ \emph{$1$-ordinary} if the equivalent conditions of Proposition~\ref{prop:1-ord} hold.
\end{defin}

\subsection{Higher Hasse--Witt operations} 

These are Frobenius-linear morphisms 
\begin{equation} \label{eqn:hwi} 
	\HW(i) \colon H^{d-i}(X, \Omega^i_X) \ra H^{d-i}(X, \Omega^i_X), \quad i = 0, \ldots, d
\end{equation}
defined by Katz \cite[2.3.4.2]{KatzAlgSoln}. The morphism $\HW(0)$ is the Hasse--Witt operation, and for $i>0$, $\HW(i)$ is only defined if the following composition is bijective
\[ 
	\xymatrix{
		h(i)\colon F_{i-1} \ar@{^{(}->}[r] & H^d_{\rm dR}(X/k) \ar@{->>}[r] & H^d_{\rm dR}(X/k)/F^i.
	}	
\]
In this case, $F_i\cap F^{i}$ projects isomorphically onto $F_i/F_{i-1}$. The mapping $\HW(i)$ is then the composition
\[ 
	H^{d-i}(X, \Omega^i_X) = F_i/F_{i-1} \isomto F_i\cap F^i \ra F^{i}/F^{i+1} = H^{d-i}(X, \Omega^i_X).
\]
Moreover, $h(i+1)$ is bijective if and only if $HW(i)$ is bijective.


\begin{prop} \label{prop:m-ord}
	Let $1\leq m\leq d+1$. The following are equivalent:
	\begin{enumerate}[(i)]
		\item The higher Hasse--Witt operations $\HW(i)$ are defined and bijective for $i<m$.
		\item $F_i \oplus F^{i+1}\isomto H^d_{\rm dR}(X/k)$ for $i<m$.
		\item The first $m$ segments of the Newton and Hodge polygons of $H^d_{\rm cris}(X/W(k))$ coincide.
		\item There exists a $\phi$-stable decomposition 
		\[ 
			H^d_{\rm cris}(X/W(k)) \isom \left(\bigoplus_{i<m} H_i\right) \oplus H_{\geq m}
		\]
		where the $H_i$ are free $W(k)$-modules of ranks $h^i = \dim_k H^{d-i}(X, \Omega^i_X)$, where $\phi|_{H_i}$ has matrix $p^i\cdot {\rm Id}$ in some basis of $H_i$, and where $\phi(H_{\geq m}) \subseteq p^m H_{\geq m}$.
	\end{enumerate}
	Moreover, these conditions hold for all $m$ if $X$ is ordinary in degree $d$ in the sense of \cite[Definition~4.12, p.\ 208]{IllusieRaynaud}.
\end{prop}

\begin{proof}
The equivalence of (i) and (ii) is clear by the definition of $\HW(i)$ and $h(i)$.  In light of \S\ref{ss:hodge_newton_polygons}, the equivalence of (ii) and (iii) is a statement purely about $F$-crystals proved e.g.\ in \cite[Prop.~1.3.2]{DeligneIllusieKatz}.  The equivalence of (iii) and (iv) follows from \cite[Theorem~1.6.1]{KatzSlopeFil}. Finally, the last assertion is \cite[4.12.1 p.\ 208]{IllusieRaynaud} (see also \cite[Proposition~7.3]{Bloch_Kato}).
\end{proof}

\begin{defin} \label{def:m-ord}
	Let $1\leq m\leq d$. We call $X$ \emph{$m$-ordinary} if the equivalent conditions of Proposition~\ref{prop:m-ord} hold.
\end{defin}

The following result will be used in \S\ref{s:nygaard}.

\begin{lemma} \label{lemma:2-ord}
	If $X$ is $2$-ordinary and $H^d(X,\Omega^1_X) = 0$, then we have $H^i(X, B\Omega^2_X) = 0$ for $i=d-1, d$. 
\end{lemma}

\begin{proof}
Since $X$ is $1$-ordinary, we know that $H^d(X,B\Omega^1_X) = 0$, and hence the short exact sequence
\[ 
	0\ra B\Omega^1_X\ra Z\Omega^1_X\ra \Omega^1_X \ra 0
\]
and the assumption $H^d(X, \Omega^1_X) = 0$ show together that $H^d(X, Z\Omega^1_X)=0$.  Then the short exact sequence
\[ 
	0\ra Z\Omega^1_X\ra \Omega^1_X\ra B\Omega^2_X \ra 0
\]
shows that $H^d(X, B\Omega^2_X)=0$. Finally, the exact triangle
\[ 
	\tau_{\leq 1}\Omega^\bullet_X \oplus \Omega^{\bullet\geq 2}_X \ra \Omega^\bullet_X\ra B\Omega^2_X[-1]\xra{+1} 
\]
gives an exact sequence
\[ 
	(F_1\oplus F^2)H^d_{\rm dR}(X/k) \isomto H^{d}_{\rm dR}(X/k)  \to H^{d-1}(X, B\Omega^2_X)
	\to (F_1\oplus F^2)H^{d+1}_{\rm dR}(X/k) \to H^{d+1}_{\rm dR}(X/k)
\]
where the first map is an isomorphism by $2$-ordinarity. It remains to show that the last map above is injective, which follows from $F_1 H^{d+1}_{\rm dR}(X/k) = H^{d+1}(X, \tau_{\leq 1} \Omega^\bullet_X) = 0$. Indeed, this group sits inside an exact sequence
\[ 
	0= H^d(X, Z\Omega^1_X) \ra H^{d+1}(X, \tau_{\leq 1} \Omega^\bullet_X) \ra H^{d+1}(X, \cO_X) = 0. \qedhere
\]
\end{proof}

\subsection{Artin--Mazur formal groups}
\label{ss:amfg}

Suppose that $H^{d-1}(X, \cO_X)=0$. In \cite{Artin_Mazur}, Artin and Mazur consider the functor 
\[ 
	\Phi_X \colon {\rm Art}_k \to {\rm Set}, 
	\quad
	\Phi_X(A) = \ker\left( H^d_\et(X_A, \GG_m)\to H^d_\et(X, \GG_m) \right) = H^d(X_A, 1 + \mathfrak{m}_A \cO_{X_A}). 
\]
and show that it is a formal Lie group. The group $\Phi_X$ is called the \emph{Artin--Mazur formal group} of $X$. Its tangent space is $H^d(X, \cO_X)$, and its (covariant) Dieudonn\'e module is canonically identified with $H^d(X, W\cO_X)$ endowed with its natural Frobenius and Verschiebung.

\begin{prop}
	The following are equivalent:
	\begin{enumerate}[(i)]
		\item $X$ is $1$-ordinary.
		\item $F\colon H^d(X, W\cO_X) \to H^d(X, W\cO_X)$ is bijective.
		\item The height of the formal group $\Phi_X$ equals one.
	\end{enumerate}
	Moreover, in this case there exists a canonical isomorphism 
	\[
		\Phi_X \isom H^d(X, W\cO_X)^{F=1} \otimes \hat \GG_m
	\]
	inducing the identity $H^d(X, \cO_X) = H^d(X, \cO_X)$ on tangent spaces.
\end{prop}

\begin{proof}
The equivalence of the three conditions follows from \cite[Theorem~2.1]{GeerKatsura}. The isomorphism
\[
  H^d(X, W\cO_X)^{F=1} \otimes W(k) \to H^d(X, W\cO_X)
 \]
is an isomorphism of Dieudonn\'e modules where on the left $F$ and $V$ act on $W(k)$ in the usual way and trivially on $H^d(X, W\cO_X)^{F=1}$.  Reversing the equivalence between suitable Dieudonne modules and formal groups, we obtain the desired isomorphism.
\end{proof}


\subsection{Ordinarity in families}
\label{ss:ord-family}

In characteristic zero, deformations of a variety with trivial canonical class have trivial canonical class. This is no longer the case in positive characteristic, for example a supersingular Enriques surface in characteristic two can be deformed to a classical one. Fortunately for us, this is true for families of $1$-ordinary varieties. In fact, such families admit a unique relative $F$-splitting, and a trivialization of the canonical bundle which is canonical up to a discrete choice (Corollary~\ref{cor:uniq-rel-fsplit} below). This will be crucial later on for the construction of the canonical lifting in families (Corollary~\ref{cor:can-lift-cy}).

\begin{prop} \label{prop:fsplit-family}
	Let $S$ be the spectrum of a noetherian local $\FF_p$-algebra with closed point $s$, and let $f:X\to S$ be a proper and flat morphism. Suppose that the maps
	\begin{equation} \label{eqn:hw-bc}
		F_{X_s}^* \colon H^i(X_s, \cO_{X_s}) \ra H^i(X_s, \cO_{X_s})
	\end{equation}
	are bijective for $i\geq 0$. Then the higher direct images $R^i f_* \cO_X$ are locally free, with formation commuting with base change, for all $i\geq 0$.
\end{prop}

\begin{proof}
The proof proceeds by descending induction on $i$. For $i>\dim X_s$, the assertion is trivial. For the induction step, suppose that $R^{i+1}f_* \cO_X$ has the required property. By cohomology and base change \cite[Theorem~12.11]{Hartshorne}, it follows that the formation of $R^i f_* \cO_X$ commutes with base change. Equivalently, the restriction map
\[ 
	(R^i f_*\cO_X)\otimes \kappa(s) \ra H^i(X_s, \cO_{X_s})
\]
is an isomorphism. Consider the commutative square
\[ 
	\xymatrix{
		R^i f_*\cO_X\ar[r] \ar[d]_F & H^i(X_s, \cO_{X_s}) \ar[d]^F \\
		R^i f_*\cO_X \ar[r]  & H^i(X_s, \cO_{X_s}). \\
	}
\]
By assumption, the right vertical arrow is an isomorphism. By the base change property of $R^i f_* \cO_X$ and Nakayama's lemma, it follows that the morphism
\[ 
	F_{X/S}^* \colon R^i f'_* \cO_{X'} \isom F_S^* R^i f_*\cO_X \ra R^i f_*\cO_X
\]
is an isomorphism. By Lemma~\ref{lemma:locfree} below, $R^i f_* \cO_X$ is then locally free, as desired.
\end{proof}

\begin{lemma} \label{lemma:locfree}
	Let $R$ be a local $\FF_p$-algebra 
	and let $M$ be a finitely presented $R$-module such that $F_R^* M \isom M$. Then $M$ is free.
\end{lemma}

\begin{proof}
Fix a presentation
\[ 
	R^m \xra{A} R^n \ra M \ra 0, \quad A = [a_{ij}] \in M_{n\times m}(R)
\]
with $n$ and $m$ minimal; in particular, $a_{ij}\in\mathfrak{m}_R$. As $F_R^* M \isom M$, the matrix $F_R(A)$ gives another minimal presentation for $M$. But every two such presentations are isomorphic, so there exist $B\in GL_n(R)$, $C\in GL_m(R)$ such that $A = B F_R(A)C$. Therefore, if  $I\subseteq \mathfrak{m}_R$ is the ideal generated by the $a_{ij}$, then $I \subseteq F_R(I)\subseteq I^2$. So $I=I^2$, and hence $I=0$.
\end{proof}



\begin{remarks}
\begin{enumerate}[1.]
	\item The assertion of Proposition~\ref{prop:fsplit-family} fails if $S$ is of mixed characteristic, already for $S=\Spec W_2(k)$. Indeed, if $p=2$ and $X/k$ is a singular Enriques surface \cite[\S II 7.3]{Illusie_deRhamWitt}, then the Frobenius acts bijectively on $H^*(X, \cO_X)$, which are one-dimensional for $*\leq 2$, and one has $H^2_{\rm cris}(X/W(k))_{\rm tors} \isom H^2(X, W\cO_X) \isom k$. Thus if $\wt X/W_2(k)$ is any lifting of $X$, then $H^2(\wt X, \cO_{\wt X}) \isom k$. Note that $X$ is a $1$-ordinary variety with trivial canonical class which satisfies (DEG) but not (NCT). 
	\item If $X_s$ is $F$-split, then the maps \eqref{eqn:hw-bc} are bijective for all $i\geq 0$ by Proposition~\ref{prop:1-ord}. 
	\item The proof of Proposition~\ref{prop:fsplit-family} shows that if the maps \eqref{eqn:hw-bc} are bijective for $i\geq r$ for some integer $r$, then the sheaves $R^i f_*\cO_X$ are locally free and commute with base change for all $i\geq r$.
\end{enumerate}
\end{remarks}

\begin{cor} \label{cor:rfomega}
	In the situation of Proposition~\ref{prop:fsplit-family}, suppose that $f$ is smooth. Then the sheaves $R^i f_* \omega_{X/S}$ are locally free, with formation commuting with base change, for all $i\geq 0$.
\end{cor}

\begin{proof}
Let $d=\dim X_s$. Grothendieck duality produces a quasi-isomorphism
\begin{align*}
	Rf_* \omega_{X/S} &= Rf_* R\cHom_X(\cO_X, \omega_{X/S})\\
	 &= Rf_* R\cHom_X(\cO_X, f^! \cO_S[-d]) \isom R\cHom_S(Rf_* \cO_X, \cO_S)[-d].
\end{align*}
Since the cohomology sheaves $R^i f_*\cO_X$ of $Rf_*\cO_X$ are locally free by Proposition~\ref{prop:fsplit-family}, we conclude that
\[ 
	R^i f_* \omega_{X/S} \isom \cHom_S(R^{d-i} f_*\cO_X, \cO_S). \qedhere
\]
\end{proof}

\begin{cor} \label{cor:uniq-rel-fsplit}
	Let $S$ be a noetherian $\FF_p$-scheme, and let $f\colon X\to S$ be a smooth and proper morphism. Suppose that for every closed $s\in S$, the geometric fiber $X_{\bar s}$ is a~$1$-ordinary variety with trivial canonical class. Then the following assertions hold.
	\begin{enumerate}[(a)]
		\item There exists a unique relative $F$-splitting $\sigma\colon F_{X/S,*}\cO_X \to \cO_{X'}$.
		\item There exists a canonical $\mu_{p-1}$-torsor $T\to S$ and a canonical trivialization $\omega_{X_T/T} \isom \cO_{X_T}$.
	\end{enumerate}
\end{cor}

\begin{proof}
We shall first prove that locally on $S$, there exists an trivialization $\omega_{X/S}\isom \cO_X$. By Corollary~\ref{cor:rfomega}, $f_* \omega_{X/S}$ is locally free, with formation commuting with base change. It is therefore a line bundle. Consider the adjunction map
\[ 
	\alpha\colon f^* f_* \omega_{X/S} \ra \omega_{X/S}.
\]
Its restriction to every fiber $X_s$ is an isomorphism by assumption. Since $\alpha$ is a morphism between line bundles, it is an isomorphism. 

For (a), consider the relative evaluation map
\[
	f'_*\cHom(F_{X/S,*}\cO_X,\cO_{X'}) \ra f'_*\cO_{X'} \ra \cO_S.
\]
It is enough to show that this map is an isomorphism. Grothendieck duality for the finite flat morphism $F_{X/S}$ yields an identification
\[ 
	f'_*\cHom(F_{X/S,*}\cO_X,\cO_{X'}) \isom f'_* F_{X/S,*}\omega_{X/S}^{1-p} = f_* \omega_{X/S}^{1-p}.
\]
By the first paragraph and Corollary~\ref{cor:rfomega}, $f_* \omega_{X/S}^{1-p}$ is a line bundle whose formation commutes with base change. The composite morphism
\[ 
	f_* \omega_{X/S}^{1-p} \isom f'_*\cHom(F_{X/S,*}\cO_X,\cO_{X'}) \ra \cO_S
\] 
is a morphism of line bundles which is an isomorphism at closed points, and hence an isomorphism.

For assertion (b), we note that the proof of (a) furnishes a canonical trivialization $\tau$ of $\omega_{X/S}^{1-p}$. Trivializations of $\omega_{X/S}$ whose $(1-p)$-th power equals $\tau$ form the desired $\mu_{p-1}$-torsor on $S$.
\end{proof}

\subsection{Hodge $F$-crystals}
\label{ss:hodgefcrystal}

Let $\wt S = \Spf W_2(k)[[t_1, \ldots, t_r]]$. The following is a version of the definition in  \cite[1.1.1, 1.1.3, 1.3.5]{DeligneIllusieKatz} adapted to the case `mod $p^2$.' Since $p^2 = 0$, our definition works well only if one is interested in the first two levels of the Hodge filtration. We assume $p>2$, so that the change of Frobenius formula \eqref{eqn:changefrob} takes a particularly simple form.

\begin{defin} \label{def:hodgefcrys}
	A \emph{Hodge $F$-crystal} $H=(H, \nabla, F^\bullet, \phi)$ on $\wt S$ consists of:
	\begin{enumerate}[(a)]
		\item a finitely generated free $\cO_{\wt S}$-module $H$,
		\item a nilpotent integrable connection $\nabla \colon H\to H\otimes \Omega^1_{\wt S}$,
		\item a decreasing filtration $F^i$ of $H$ by direct summands,
		\item for every lifting of Frobenius $F_{\wt S}$ on $S$, an $\cO_{\wt S}$-linear map
		\[
			\phi(F_{\wt S})\colon F_{\wt S}^* H \ra H
		\]
	\end{enumerate}
	satisfying the following conditions:
	\begin{enumerate}[(i)]
		\item $\phi(F_{\wt S})$ is horizontal,
		\item $F^i$ satisfies \emph{Griffiths transversality}: $\nabla F^i \subseteq F^{i-1}\otimes \Omega^1_{\wt S}$,
		\item $F^i \otimes k$ coincides with the image of $(\phi(F_{\wt S})F_{\wt S}^*)^{-1} (p^i H)$ for $i=0,1,2$,
		\item the maps $\phi(F_{\wt S})$ for different $F_{\wt S}$ satisfy the \emph{change of Frobenius formula} \cite[1.1.3.4]{DeligneIllusieKatz}:
		\begin{equation} \label{eqn:changefrob} 
			\phi(F_{\wt S})F_{\wt S}^* x 
			= 
			\phi(F'_{\wt S})(F'_{\wt S})^* x 
			+
			p\sum_{i=1}^r (F_{\wt S}^*(t_i) - (F'_{\wt S})^*(t_i)) \phi(F'_{\wt S})(F'_{\wt S})^*(\nabla_{\frac{\partial}{\partial t_i}} x).
		\end{equation}
	\end{enumerate}
\end{defin}

If $\wt X/\wt S$ is smooth and projective satisfying (NCT) and (DEG) and such that the Hodge groups are free, then $H= H^n_{\rm dR}(\wt X/\wt S)$ endowed with the Gauss--Manin connection, the Hodge filtration, and the crystalline Frobenius is a Hodge $F$-crystal over $\wt S$. (All statements except (iii) are standard.  For (iii), we observe that the statement only concerns the Frobenius on the crystalline cohomology of the reduction mod $p$.  Therefore we may apply the results of Mazur--Ogus \cite[8.26]{BerthelotOgus} mentioned in \S\ref{ss:hodge_newton_polygons}, valid over a torsion-free base, and then reduce them mod $p^2$.)  Note that for $H$ we also have the basic divisibility estimates 
\[ 
	\phi(F_{\wt S})F_{\wt S}^* (F^1) \subseteq p H, \quad \phi(F_{\wt S})F_{\wt S}^* (F^2) \subseteq p^2 H \text{ if }p>2.
\]

\section{Relative Witt vectors}
\label{s:relwitt}


\noindent
Let $S$ be a scheme over ${\bf F}_p$, and let $\wt S$ be a lifting of $S$ over $\ZZ/p^2\ZZ$. Consider a smooth scheme $X$ over $S$. 
We let $F_{X/S}\colon X\to X'$ denote the relative Frobenius of $X/S$. If $S$ is perfect, then using the isomorphism $F_S^*\colon \cO_X\isomto \cO_{X'}$ one can consider the length two Witt vectors $W_2\cO_X$ as a square-zero extension
\[ 
	0\ra F_{X/S, *} \cO_X \xra{V} W_2 \cO_X \ra \cO_{X'} \ra 0
\]
of $\cO_{X'}$ by $F_{X/S,*} \cO_X$. The first goal of this section is to construct a natural extension of the above type 
\begin{equation} \label{eqn:relw2type}
	0\ra F_{X/S, *} \cO_X \xra{V} W_2 (\cO_X/\wt S) \ra \cO_{X'} \ra 0
\end{equation}
for a general $\wt S$, lying above the given lifting
\[
	0\ra \cO_S \ra \cO_{\wt S} \ra \cO_S \ra 0.
\]



The second goal concerns comparison with zeroth crystalline cohomology. Let $\wt X$ be a~lifting of $X$ over $\wt S$. We have a canonical isomorphism \cite[\S 7]{BerthelotOgus}
\[ 
	\mathscr{H}^q_{dR}(\wt X/\wt S) \isom   R^q u_* \cO_{X/\wt S}
\]
where $u\colon (X/\wt S)_{\rm cris} \to X_{\rm Zar}$ is the canonical projection, and $\mathscr{H}^q_{dR}(\wt X/\wt S)$ is the $q$-th cohomology sheaf of the de Rham complex of $\wt X/\wt S$. If, in addition, $S$ is perfect, one has Katz's isomorphism \cite[III 1.4]{IllusieRaynaud}
\[
  R^q u_* \cO_{X/\wt S} \isom W_2\Omega^q_X
\]
where $W_2\Omega^\bullet_X$ is the de Rham--Witt complex of $X$. Setting $q=0$, we obtain an identification
\[ 
	W_2\cO_X = \cO_{\wt X}^{d=0} = \{ f\in \cO_{\wt X} \,|\, df = 0 \in \Omega^1_{\wt X/\wt S}\}.
\]
The right hand side makes sense if $S$ is not perfect, so one could try and define
\[ 
	W_2 (\cO_X/\wt S) =  \cO_{\wt X}^{d=0}  = u_* \cO_{X/\wt S}
\]
which by the last equality is manifestly independent of the choice of the lifting $\wt X$. One quickly realizes however that $\cO_{\wt X}^{d=0}$ is not quite of the required type, and that one in fact gets an extension
\[ 
	0\ra \cO_{X'} \ra \cO_{\wt X}^{d=0} \ra \cO_{X''} \ra 0
\]
with a double Frobenius twist on the right hand side (compare \cite[p.\ 103]{Olsson}). For example, if $\wt X = \mathbf{A}^1_{\wt S}$ with coordinate $t$, then one obtains
\[ 
	0\ra \cO_S[x^{p}]\xra{p} \cO_{\wt S}[x^{p^2}, px^{pi}] \ra \cO_S[x^{p^2}]\ra 0.
\]
We shall prove in \S\ref{ss:h0cris} that $W_2(\cO_{X}/\wt S)$ is a canonical Frobenius untwist of $u_* \cO_{X/\wt S}$, i.e.\ that $u_* \cO_{X/\wt S} \isom W_2(\cO_{X'}/\wt S)$.


\begin{remark}
By the same method as given below, one could define longer Witt vectors relative to a lifting of $S$ modulo $p^n$. For the sake of brevity, we decided to include only the case relevant to our applications.
\end{remark}

\subsection{Construction}

Let $X$ be a smooth scheme over $S$ as before, and let $\wt X$ be a lifting over $\wt S$. We set
\begin{equation} \label{eqn:defwtilde}
	W_{\wt X} = \{ f \in \cO_{\wt X} \,|\, f\, {\rm mod}\, p \in \cO_{X'}  \} \subseteq \cO_{\wt X}. 
\end{equation}
In other words, $W_{\wt X}$ is the pullback
\[ 
	\xymatrix{
		W_{\wt X} \ar[r] \ar[d] \ar@{}[dr]|-{\square} &  \cO_{X'} \ar[d]^{F_{X/S}^*} \\
		 \cO_{\wt X}  \ar[r] & \cO_X,
	}
\]
so that we have a diagram with exact rows (morphism of square-zero extensions)
\[
	\xymatrix{
		0 \ar[r] & \cO_X \ar[r]\ar@{=}[d] & W_{\wt X}\ar[r] \ar[d] & \cO_{X'} \ar[r]\ar[d]^{F_{X/S}^*} & 0 \\
		0 \ar[r] & \cO_X \ar[r] & \cO_{\wt X}\ar[r] & \cO_X \ar[r] & 0.
	}
\]
In particular, $W_{\wt X}$ is an extension of the desired type \eqref{eqn:relw2type}. Our goal is to get rid of its apparent dependence on the lifting $\wt X$. 

We note that $f\in \smash{\cO_{\wt X}}$ lies in $W_{\wt X}$ if and only if $df\in \smash{\Omega^1_{\wt X/\wt S}}$ is divisible by $p$, so that we also have a pullback square
\[ 
	\xymatrix{
		W_{\wt X} \ar[r] \ar[d] \ar@{}[dr]|-{\square} & \Omega^1_{X/S}  \ar[d]^p \\
		\cO_{\wt X} \ar[r]_-d & \Omega^1_{\wt X/\wt S}.
	}
\]

\begin{lemma}\label{lem:fundamental}
	Let $f\colon X\to Y$ be a map between smooth schemes over $S$, let $\wt X$ and $\wt Y$ be liftings of $\wt X$ and $\wt Y$ over $\wt S$, and let $\tilde f_1, \tilde f_2 \colon \wt X \to \wt Y$ be two liftings of $f$.  Then the induced maps $\tilde f_1, \tilde f_2 \colon W_{\wt X} \to W_{\wt Y}$ are equal.  
	
	In particular, if $\alpha\colon \wt X\to \wt X$ is an automorphism of the lifting $\wt X$ (i.e. an isomorphism of schemes over $\wt S$ reducing to the identity on $X$), then the induced map $\alpha\colon W_{\wt X} \to W_{\wt X}$ is the identity.
\end{lemma}

\begin{proof}
We can write $\tilde f_2(x) = \tilde f_1(x) + p\cdot \psi(x\,{\rm mod}\,p)$ for a derivation $\psi\colon \cO_Y \to f_*\cO_X$. Since $F_{Y/S}$ induces the zero map $F_{Y/S}^* \Omega^1_{Y'/S}\to \Omega^1_{Y/S}$, the derivation $\psi$ annihilates the image of $\cO_{Y'}$. Thus if $x \in W_{\wt Y}$, i.e. $x \,{\rm mod}\, p \in \cO_{Y'}$, then
\[
	\tilde f_2(x) = \tilde f_1(x) + p\cdot \psi(x\,{\rm mod}\, p) = \tilde f_1(x). \qedhere
\]
\end{proof}

The above lemma allows us to assemble the $W_{\wt X}$ into a global thickening $W_2(\cO_X/\wt S)$ in the following way. Choose an open cover $X=\bigcup U_i$ and liftings $\wt U_i$ over $\wt S$. Let $\{U_{ijk}\}_k$ be an affine open cover of $U_i\cap U_j$. Since $U_{ijk}$ are affine, there exists an isomorphism of liftings 
\[ 
	\alpha_{ijk} \colon \wt U_i|_{U_{ijk}} \ra \wt U_j|_{U_{ijk}}.
\]
Consider the schemes $W_{\wt U_i}$; over $U_{ijk}$, the above maps induce isomorphisms
\[ 
	\alpha_{ijk} \colon W_{\wt U_i}|_{U_{ijk}} \ra W_{\wt U_j}|_{U_{ijk}}.
\]
which are independent of $\alpha_{ijk}$ by Lemma~\ref{lem:fundamental}. Another application of Lemma~\ref{lem:fundamental} implies that these isomorphisms necessarily satisfy the cocycle condition, and therefore they can be glued to obtain a global thickening $W_2(\cO_X/\wt S)$.

It is easy to see that this construction does not depend on the choice of the open covering and the liftings $\wt U_i$. We have thus proved:

\begin{prop}
	Let $X$ be a smooth scheme over $S$. The construction outlined above yields a thickening
	\[
		0\ra F_{X/S,*} \cO_X \ra W_2(\cO_X/\wt S) \ra \cO_{X'} \ra 0	
	\]
	over
	\[
		0\ra \cO_S \ra \cO_{\wt S}\ra \cO_S \ra 0.	
	\]
	If $\wt X$ is a lifting of $X$ over $\wt S$, there is a canonical isomorphism 
	\[
		W_2(\cO_X/\wt S) \isom W_{\wt X} = \{ f \in \cO_{\wt X} \,|\, f\, {\rm mod}\, p \in \cO_{X'}  \} \subseteq \cO_{\wt X}.
	\]
	Moreover, the construction of $W_2(\cO_X/\wt S)$ is functorial in the sense that if $f\colon X\to Y$ is a morphism of smooth schemes over $S$, then we get a canonical map $f^*\colon W_2(\cO_Y/\wt S)\to W_2(\cO_X/\wt S)$ fitting in the diagram of sheaves on the space $Y=Y'$
	\[
		\xymatrix{
			0\ar[r] & F_{Y/S, *} \cO_Y \ar[d]\ar[r] & W_2(\cO_Y/\wt S) \ar[r]\ar[d] & \cO_{Y'} \ar[r]\ar[d] & 0 \\
			0\ar[r] & f_* F_{X/S, *} \cO_X \ar[r] & f_* W_2(\cO_X/\wt S) \ar[r] & f_* \cO_{X'} \ar[r] & 0.
		}	
	\]
\end{prop}

\subsection{Comparison with zeroth crystalline cohomology}
\label{ss:h0cris}

Let $X$ be a smooth scheme over $S$. We denote by
\[ 
	u \colon (X/\wt S)_{\rm cris}\ra X_{\rm Zar}
\]
the natural projection from the crystalline site of $X$ relative to $\wt S$ to the Zariski site of $X$. This functor takes an open $U\subseteq X$ to the trivial PD-thickening $(U, U)$. If $\wt X$ is a lifting of $X$ over $\wt S$, there is a canonical isomorphism \cite[\S 7]{BerthelotOgus}
\[ 
	Ru_* \cO_{X/\wt S} \isom \Omega^\bullet_{\wt X/\wt S}.
\]
We denote by $\mathscr{H}^0_{dR}(\wt X/\wt S)$ the kernel of the map $d\colon \cO_{\wt X}\to \smash{\Omega^1_{\wt X/\wt S}}$. The restriction map $\mathscr{H}^0_{dR}(\wt X/\wt S)\to \mathscr{H}^0_{dR}(X/S) = \cO_{X'}$ has image $\cO_{X''}$ and the following sequence is exact
\[ 
	0\ra F_{X'/S,*}\cO_{X'} \xra{\text{``$p$''}} \mathscr{H}^0_{\rm dR}(\wt X/\wt S) \ra \cO_{X''} \ra 0.
\]

\begin{prop}
	There exists a unique morphism of sheaves of rings 
	\[
	\lambda \colon W_2(\cO_{X'}/\wt S) \ra u_*\cO_{X/\wt S}
	\]
	compatible with restriction to open subsets, and such that whenever $\wt F \colon \wt X \to \wt X{}'$ is a lifting of $F_{X/S} \colon X \to X'$, then the diagram
	\[
		\xymatrix{
		\cO_{\wt X{}'} \ar[r]^{\wt F{}^*} & \cO_{\wt X} \\
		W_{\wt X{}'} \ar[r]^{\wt F{}^*}\ar@{^{(}->}[u] & \mathscr{H}^0_{\rm dR}(\wt X/\wt S) \ar@{^{(}->}[u] \\
		W_2(\cO_{X'}/\wt S) \ar[u]^{\rotatebox{90}{$\sim$}} \ar[r]_{\lambda}^{\sim} & u_*\cO_{X/\wt S} \ar[u]_{\rotatebox{90}{$\sim$}}
		}
	\]
	is commutative.  This map is an isomorphism.
\end{prop}

\begin{proof}
We define the morphism $\lambda$ locally by the diagram above, and then prove this is independent of the choices.  To this end, we first take two liftings $\wt F_1, \wt F_2 \colon \wt X \to \wt X{}'$ of the Frobenius morphism $F_{X/S}$.  By Lemma~\ref{lem:fundamental} the morphisms $\wt F{}_1^*, \wt F{}_2^* \colon W_{\wt X{}'} \to W_{\wt X}$ coincide.  We check that they map $W_{\wt X{}'}$ onto $\mathscr{H}^0_{\rm dR}(\wt X/\wt S) \subset W_{\wt X}$: if $f\in W_{\wt X{}'}$, then $df = p\cdot \omega$ for some $\omega\in \Omega^1_{X'/\wt S}$, therefore
\[ 
	d(\wt F_i{}^*(f)) = \wt F_i{}^*(df) = p\cdot F_{X'/S}^* (\omega) = 0.
\]  
Finally, since the following diagram is obviously commutative
\[
\xymatrix{
	0 \ar[r] & F_{X'/S,*}\cO_{X'} \ar@{=}[d]\ar[r] & W_{\wt X{}'} \ar[r]\ar[d]^{\wt F_i{}^*} & \cO_{X''} \ar[r]\ar@{=}[d] & 0 \\
	0 \ar[r] & F_{X'/S,*}\cO_{X'} \ar[r] & \mathscr{H}^0_{\rm dR}(\wt X/\wt S) \ar[r] & \cO_{X''} \ar[r] & 0, 
}
\] 
we conclude that $W_{\wt X{}'}\to \mathscr{H}^0_{\rm dR}(\wt X/\wt S)$ is an isomorphism.
\end{proof}

\begin{cor}
	If $S$ is perfect, then $W_2(\cO_X/\wt S) \isom W_2\cO_X$.
\end{cor}





\section{\texorpdfstring{$F$-splittings and canonical liftings mod $p^2$}{F-splittings and canonical liftings mod p2}}
\label{s:canlift}

\noindent
We come to the key construction in this paper.

\subsection{Construction}
\label{ss:cons-canlift}

Let $S$ be an $\FF_p$-scheme, and let $\wt S$ be a flat lifting of $S$ over $\ZZ/p^2\ZZ$. Let $(X,\sigma)$ be a pair consisting of an $S$-scheme $X$ and a relative $F$-splitting
\[ 
	\sigma\colon F_{X/S,*}\cO_X \ra \cO_{X'}.
\]
Such pairs form a category $\mathbf{FSplit}_S$, cf.\ \cite[\S 2.5]{AchingerWitaszekZdanowicz}. Consider the ring of relative Witt vectors $W_2(\cO_X/\wt S)$, fitting in an exact sequence
\[ 
	0\ra F_{X/S,*}\cO_X \xra{V} W_2(\cO_X/\wt S) \ra \cO_{X'}\ra 0.
\] 
The subsheaf
\[ 
	I_\sigma = V(\ker(\sigma)) \subseteq W_2(\cO_X/\wt S).
\]
is an ideal in $W_2(\cO_X/\wt S)$ --- this is because $V(F_{X/S,*}\cO_X)$ is an ideal of square zero inside $W_2(\cO_X/\wt S)$, and $\ker(\sigma)\subseteq F_{X/S,*}\cO_X$ is an $\cO_{X'}$-submodule. We define the sheaf of rings $\cO_{\wt X(\sigma)}$ as the quotient
\[ 
	\cO_{\wt X(\sigma)} = W_2(\cO_X/\wt S) / I_\sigma.
\]
By construction, it fits inside a commutative diagram with exact rows
\[ 
	\xymatrix{
		0\ar[r] & F_{X/S,*}\cO_X \ar[r]^V \ar[d]_\sigma & W_2(\cO_X/\wt S) \ar[r] \ar[d] & \cO_{X'} \ar@{=}[d] \ar[r] & 0 \\
		0\ar[r] & \cO_{X'} \ar[r] & \cO_{\wt X(\sigma)} \ar[r] & \cO_{X'} \ar[r] & 0
	}
\]
lying over
\[ 
	\xymatrix{
		0\ar[r] & \cO_S \ar[r] & \cO_{\wt S} \ar[r] & \cO_S\ar[r] & 0.
	}
\]
We deduce the following.

\begin{prop} \label{prop:can-lift}
	The ringed space $\wt X(\sigma) = (X, \cO_{\wt X(\sigma)})$ is a scheme, and the surjection $\cO_{\wt X(\sigma)}\to \cO_{X'}$ exhibits $\wt X(\sigma)$ as a flat lifting of $X'$ over $\wt S$. The construction defines a~functor
	\[ 
		(X, \sigma)\mapsto \wt X(\sigma) \colon \mathbf{FSplit}_S \ra \mathbf{Sch}_{\wt S}
	\] 
	together with a natural isomorphism between the two compositions in the square below
	\[ 
		\xymatrix{
			\mathbf{FSplit}_S \ar[r] \ar[d]_{\mathrm{forget}} & \mathbf{Sch}_{\wt S} \ar[d]^{-\times_{\wt S} S} \\
			\mathbf{Sch}_S \ar[r]_{F_S^*} & \mathbf{Sch}_S,
		}
	\]
	i.e.\ $\wt X(\sigma)\times_{\wt S} S \isom X'$.
\end{prop}

\begin{remark}
One can check that $\Spec W_2(\cO_X/\wt S)$ represents the following functor on the category of square-zero thickenings of $X'$ over $\wt S$ and closed immersions:
\[ 
	0\to M\to \cO_{Y} \to \cO_{X'}\to 0 \quad \mapsto \quad \{\text{surjections }\alpha\colon F_{X/S,*}\cO_X\to M\,|\, \alpha(1) = p \}.
\]
In particular $F$-splittings of $X/S$ are in natural bijection with $\wt S$-liftings of $X'$ endowed with a closed immersion to $\Spec W_2(\cO_X/\wt S)$ over $\wt S$. This should be contrasted with the fact that (for $S$ perfect) maps \emph{from} $\Spec W_2(\cO_X/\wt S)$ control Frobenius liftings.
\end{remark}

Combining this with Corollary~\ref{cor:uniq-rel-fsplit}(a), we obtain:

\begin{cor} \label{cor:can-lift-cy}
	The canonical lifting functor from Proposition~\ref{prop:can-lift} produces a functor $X\mapsto \wt X$ from smooth $X/S$ whose geometric fibers are $1$-ordinary varieties with trivial canonical class to smooth $\wt X/\wt S$ with the same property, together with a functorial isomorphism $\wt X\times_{\wt S} S \isom X'$.
\end{cor}

\subsection{First properties}

As a consequence of functoriality of $\wt X(\sigma)$, we obtain the following.

\begin{lemma}
	Let $S$ and $\wt S$ be as in \S\ref{ss:cons-canlift}, and let $(X, \sigma)$ be an $F$-split scheme over $S$. Let $Y\subseteq X$ be a closed subscheme which is compatible with $\sigma$ in the sense that $\sigma(F_{X/S,*} I_Y)\subseteq I_Y'$. Then $\sigma$ induces an $F$-splitting $\sigma_Y$ on $Y$ relative to $S$, and $\wt Y(\sigma_Y)$ is a closed subscheme of $\wt X(\sigma)$ lifting $Y'\subseteq X'$.
\end{lemma}

We shall not use the above result, as the unique $F$-splitting on a $1$-ordinary variety with trivial canonical class cannot be compatible with a non-empty proper subvariety. In contrast, the following will be quite useful.

\begin{lemma} \label{lemma:teichmueller}
	Let $S$ be a perfect scheme over $\FF_p$ and let $\wt S = W_2(S)$. Let $(X, \sigma)$ be an $F$-split scheme over $S$. Then the composition of the Teichm\"uller lift 
	\[ 
		[-]\colon \cO_{X'} \ra W_2(\cO_X) = W_2(\cO_X/\wt S)
	\]
	with the restriction map $W_2(\cO_X)\to \cO_{\wt X(\sigma)}$ yields multiplicative sections of the restriction maps 
	\[ 
		\cO_{\wt X(\sigma)}\ra \cO_{X'}
		\quad\text{and}\quad
		\cO_{\wt X(\sigma)}^\times\ra \cO_{X'}^\times.
	\]
	Consequently, for every $r\geq 0$, the restriction map
	\[ 
		H^r_\et(\wt X(\sigma), \GG_m) \ra H^r_\et(X', \GG_m)
	\]
	admits a natural section. In particular, line bundles and log structures on $X'$ have canonical liftings to $\wt X(\sigma)$.
\end{lemma}

\begin{lemma} \label{lemma:hw-ho}
	Let $S$ be a noetherian $\FF_p$-scheme and let $\wt S$ be a flat lifting of $S$ over $\ZZ/p^2\ZZ$. Let $(X, \sigma)$ be an $F$-split scheme over $S$ such that $X$ is proper over $S$. Then the restriction maps
	\[ 
		i^* \colon H^*(X, W_2(\cO_X/\wt S)) \ra H^*(\wt X, \cO_{\wt X})
	\]
	are isomorphisms.
\end{lemma}

\begin{proof}
The restriction map $i^*\colon W_2(\cO_X/\wt S)\to \cO_{\wt X}$ fits into a short exact sequence
\[ 
	0\ra B\Omega^1_{X/S} \ra W_2(\cO_X/\wt S) \ra \cO_{\wt X} \ra 0.
\]
By Proposition~\ref{prop:1-ord}, we have $H^*(X_s, B\Omega^1_{X_s}) = 0$ for all $s\in S$, and hence $H^*(X, B\Omega^1_{X/S}) = 0$ by \cite[Theorem~12.11]{Hartshorne} and the Leray spectral sequence for $X/S$.
\end{proof}

\begin{lemma}
	Let $X/k$ be a $1$-ordinary variety with trivial canonical class, and let $\wt X / W_2(k)$ be its canonical lifting. Suppose that $H^d(X, W\cO_X)$ is torsion-free. Then $\omega_{\wt X/W_2(k)} \isom \cO_{\wt X}$.
\end{lemma}

\begin{proof}
The assumption on torsion implies that $H^d(X, W_2\cO_X)$ is a free $W_2(k)$-module. On the other hand, by Lemma~\ref{lemma:hw-ho}, we have $H^d(X, W_2\cO_X)\isom H^d(\wt X, \cO_{\wt X})$, thus the latter is locally free. Since $d=\dim X$, the formation of $R^d f_* \cO_{\wt X}$ commutes with base change. The rest of the argument follows the lines of proof of Corollary~\ref{cor:rfomega} (for $i=0$).
\end{proof}


\section{Frobenius and the Hodge filtration (I)}
\label{s:fhodge1}

\noindent
Let $S$ be a noetherian affine $k$-scheme and let $\wt S$ be a flat lifting of $S$ over $W_2(k)$. Let $(X, \sigma)$ be a smooth projective $F$-split scheme over $S$ 
endowed with an $F$-splitting $\sigma$. We endow $\wt S$ with the natural PD-structure on the ideal $(p)$. The key property of the canonical lifting $\wt X(\sigma)/\wt S$ defined in the previous section, which will enable us to get our hands on the associated crystal, is the following.

\begin{thm} \label{thm:fhodge1}
	For every lifting $F_{\wt S}\colon \wt S\to \wt S$ of the Frobenius of $S$, the crystalline Frobenius 
	\[
		\phi(F_{\wt S})\colon H^n_{\rm dR}(\wt X(\sigma)/\wt S) \ra H^n_{\rm dR}(\wt X(\sigma)/\wt S)
	\]
	maps $F^1 H^n_{\rm dR}(\wt X(\sigma)/\wt S)$ into itself for every $n\geq 0$
\end{thm}

\noindent The proof of Theorem~\ref{thm:fhodge1} occupies \S\ref{ss:fhodge1pf-start}--\ref{ss:fhodge1pf-end}. In \S\ref{ss:fhodge1-drw}, we relate the Hodge filtration to the Hodge--Witt filtration in the case $S=\Spec k$, which is used to characterize the canonical lifting in \S\ref{ss:fhodge1-uniq}.  In the subsequent \S\ref{ss:fhodge1-hw1}, we apply this theorem to relate $\phi$ to the first higher Hasse--Witt operation. In Section~\ref{ss:fhodge2}, we will revisit this idea for the modular lifting of Frobenius acting on the crystal associated to a universal deformation of a $1$-ordinary variety with trivial canonical class.



\begin{example}
	If $\wt X$ is a smooth and projective scheme over $\wt S = \Spec W_2(k)$ admitting a lift of Frobenius $\wt F$, then the crystalline Frobenius $\phi$ coincides with the map 
	\[ 
		\wt F{}^* \colon  H^*_{\rm dR}(\wt X/\wt S) \ra  H^*_{\rm dR}(\wt X/\wt S)
	\]
	and hence it preserves the Hodge filtration. In fact, $\wt X{}'$ is canonically isomorphic to $\wt X(\sigma)$ for any $F$-splitting $\sigma$, see \cite[Theorem~3.6.5]{AchingerWitaszekZdanowicz}.
\end{example}

\subsection{Proof of Theorem~\ref{thm:fhodge1}: first steps}
\label{ss:fhodge1pf-start}

Let $\wt X = \wt X(\sigma)$ and $\phi = \phi(F_{\wt S})$ for brevity, and denote by
\[ 
	i \colon \wt X \hookrightarrow W_2(X/\wt S) = \Spec W_2(\cO_X/\wt S)
\] 
the canonical closed immersion. 
By the definition of $F^1$, the projection $\smash{\Omega^\bullet_{\wt X/\wt S}}\to \cO_{\wt X}$ induces an injection
\[ 
	H^n_{\rm dR}(\wt X/\wt S)/F^1 \hookrightarrow H^n(\wt X, \cO_{\wt X}).
\]
To prove the required assertion, it thus suffices to produce a dotted arrow below making the square commute
\begin{equation} \label{eqn:square}
	\xymatrix{
		H^n_{\rm dR}(\wt X/\wt S) \ar[r] \ar[d]_{\phi} & H^n(\wt X, \cO_{\wt X}) \ar@{.>}[d] \\
		H^n_{\rm dR}(\wt X/\wt S) \ar[r]  & H^n(\wt X, \cO_{\wt X}). 
	}
\end{equation}
On the other hand, by Lemma~\ref{lemma:hw-ho}, the map $i^*$ induces an isomorphism
\[ 
	i^*\colon H^n(X, W_2(\cO_X/\wt S)) \isomlong  H^n(\wt X, \cO_{\wt X}),
\]
while the pair $(F_X, F_{\wt S})$ induces by functoriality of $W_2(\cO_X/\wt S)$ a map
\[ 
	F \colon H^n(X, W_2(\cO_X/\wt S)) \ra H^n(X, W_2(\cO_X/\wt S)).
\]
Consequently, it suffices to construct a morphism in the derived category
\begin{equation} \label{eqn:mapt} 
	t\colon \Omega^\bullet_{\wt X/\wt S} \ra W_2(\cO_X/\wt S)
\end{equation}
such that the diagrams
\begin{equation} \label{eqn:fhodge1}
	\xymatrix{
		\Omega^\bullet_{\wt X/\wt S} \ar[r]^-t \ar[d]_{\phi} & W_2(\cO_X/\wt S) \ar[d]^{F} \\
		\Omega^\bullet_{\wt X/\wt S} \ar[r]_-t  & W_2(\cO_X/\wt S) .
	}
\end{equation}
and
\begin{equation} \label{eqn:fhodge2}
	\xymatrix{
		& W_2(\cO_X/\wt S)  \ar[d]^{i^*} \\
		\Omega^\bullet_{\wt X/\wt S} \ar[ur]^t \ar[r]  & \cO_{\wt X}.
	}
\end{equation}
commute. This will be our strategy for the proof in \S\ref{ss:fhodge1pf-cris}--\ref{ss:fhodge1pf-end}.

\subsection{Review of crystalline theory}
\label{ss:fhodge1pf-cris}

To construct \eqref{eqn:mapt} and prove the commutativity of \eqref{eqn:fhodge1}, we first need to explicate the construction of the crystalline Frobenius $\phi(F_{\wt S})$ (e.g.\ \cite[0 3.2]{Illusie_deRhamWitt}). 

Let $e\colon \wt X\to \wt Y$ be a closed immersion into a smooth $\wt S$-scheme $\wt Y$, and let
$\bar Y$ be the PD-envelope of $e$. Then the map induced by $\bar e\colon \wt X\to \bar Y$,
\[ 
	\Omega^\bullet_{\wt Y/\wt S}|_{\bar Y} \ra \Omega^\bullet_{\wt X/\wt S},
\]
is a quasi-isomorphism. On the other hand, $\bar Y$ coincides with the PD-envelope of $X$ in $\wt Y$. Consequently, 
if now $F_{\wt Y}\colon \wt Y\to \wt Y$ is a lifting of Frobenius compatible with $F_{\wt S}$, then by functoriality of the PD-envelope $F_{\wt X}$ naturally descends to a map $F_{\bar Y}\colon \bar Y\to \bar Y$. The crystalline Frobenius $\phi(F_{\wt S})$ is the composition
\[ 
	\phi(F_{\wt S})\colon
	\Omega^\bullet_{\wt X/\wt S} \isom \Omega^\bullet_{\wt Y/\wt S}|_{\bar Y} \xra{F_{\wt Y}^*} \Omega^\bullet_{\wt Y/\wt S}|_{\bar Y} \isom \Omega^\bullet_{\wt X/\wt S}.
\]
It is independent of the choice of $(\wt Y, e, F_{\wt Y})$. Such data exist if $\wt X$ is projective (take $\wt Y$ to be a projective space and $F_{\wt Y}$ be the map raising the coordinates to the $p$-th power).

\subsection{Construction of the map $t$ \eqref{eqn:mapt}}
\label{ss:fhodge1pf-const}

Let $\wt Y$ be a smooth $\wt S$-scheme endowed with a~Frobenius lifting $F_{\wt Y}$ commuting with $F_{\wt S}$. By definition, $W_2(\cO_Y/\wt S)$ is identified with the subsheaf $W_{\wt Y}$ \eqref{eqn:defwtilde} of $\cO_{\wt Y}$, and the image of $F_{\wt Y}^* \colon \cO_{\wt Y} \ra \cO_{\wt Y}$ is contained in $W_{\wt Y}$. We obtain a natural map
\[ 
	t(F_{\wt Y}) \colon \cO_{\wt Y} \ra W_2(\cO_Y/\wt S).
\]
The following square commutes
\begin{equation} \label{eqn:tf-ft}
	\xymatrix{
		\cO_{\wt Y} \ar[d]_{F_{\wt Y}} \ar[r]^-{t(F_{\wt Y})} & W_2(\cO_Y/\wt S) \ar[d]^{F_Y} \\
		\cO_{\wt Y} \ar[r]_-{t(F_{\wt Y})} &  W_2(\cO_Y/\wt S).
	}
\end{equation}
If $S=\Spec k$, this is the Cartier morphism \cite[0 1.3.21]{Illusie_deRhamWitt}. 

Let now $e\colon \wt X\to \wt Y$ be a closed immersion into a smooth $\wt S$-scheme and let $\bar Y$ be the PD-envelope of $e$, which is the same as the PD-envelope of $X$ in $\wt Y$. The closed immersion $X\hookrightarrow W_2(X/\wt S)$ is a PD-thickening of $X$ over $S\hookrightarrow \wt S$ fitting inside a commutative diagram of solid arrows
\begin{equation} \label{eqn:pd-fact}
	\xymatrix{
		X\ar[d] \ar[r] & \bar Y \ar[r] & \wt Y \\ 
		W_2(X/\wt S) \ar[rr] \ar@{.>}[ur]^{\bar t} &  & W_2(Y/\wt S) \ar[u]^{t(F_{\wt Y})}.
	}
\end{equation}
By the universal property defining the PD-envelope, we obtain a unique dotted arrow making the above diagram commute. We define now the desired map $t$ as the composition
\begin{equation} \label{eqn:def-t} 
	t\colon \Omega^\bullet_{\wt X/\wt S} \isom \Omega^\bullet_{\wt Y/\wt S}|_{\bar Y} \ra \cO_{\bar Y} \xra{\bar t^*} W_2(\cO_X/\wt S). 
\end{equation}

\subsection{Commutativity of \eqref{eqn:fhodge1}} 
\label{ss:fhodge1pf-comm1}

By the description of the crystalline Frobenius \S\ref{ss:fhodge1pf-cris}, the left square below commutes
\[
	\xymatrix{
		\Omega^\bullet_{\wt X/\wt S} \ar[d]_{\phi(F_{\wt S})} & \Omega^\bullet_{\wt Y/\wt S}|_{\bar Y} \ar[l]_\sim \ar[r] \ar[d]^{F_{\bar Y}^*} & \cO_{\bar Y} \ar[r]^-{\bar t^*} \ar[d]^{F_{\bar Y}^*} & W_2(\cO_X/\wt S) \ar[d]^{F_X} \\
		\Omega^\bullet_{\wt X/\wt S} & \Omega^\bullet_{\wt Y/\wt S}|_{\bar Y} \ar[l]_\sim \ar[r] & \cO_{\bar Y} \ar[r]^-{\bar t^*} & W_2(\cO_X/\wt S) 
	}
\]
Since the horizontal compositions equal $t$, we are left with checking the commutativity of the remaining two squares. The middle square commutes by the definition of $F_{\bar Y}$. For the right square, note that by the commutativity of \eqref{eqn:tf-ft} we get a morphism of solid diagrams \eqref{eqn:pd-fact}, and hence a morphism of diagrams with the dotted arrows, i.e.\ the commutativity in question.

\subsection{Commutativity of \eqref{eqn:fhodge2}}
\label{ss:fhodge1pf-end}

The diagram in the derived category can be described in terms of maps of complexes as
\begin{equation} \label{eqn:small-triangle}
	\xymatrix{
		\Omega^\bullet_{\wt Y/\wt S}|_{\bar Y} \ar[d]_{\rotatebox{90}{$\sim$}} \ar[r] & \cO_{\bar Y} \ar[r]^-{\bar t^*} \ar[dr]_{\bar e^*} & W_2(\cO_X/\wt S) \ar[d]^{i^*} \\
		\Omega^\bullet_{\wt X/\wt S} \ar[rr] & & \cO_{\wt X} 
	}
\end{equation}
where the left cell clearly commutes. Interestingly, the right cell commutes on the level of sheaves only for certain lifts of Frobenius $F_{\wt Y}$.

\begin{defin}
	Let $(\wt Y, e, F_{\wt Y})$ be as in \S\ref{ss:fhodge1pf-cris}. We call such data \emph{retracting} if the diagram below commutes
	\[ 
		\xymatrix{
			\wt X \ar[d]_i \ar[r]^e & \wt Y \\
			W_2(X/\wt S) \ar[r]_e & W_2(Y/\wt S). \ar[u]_{t(F_{\wt Y})}
		}
	\]
\end{defin}

\noindent
It is easy to construct retracting triples $(\wt Y, e, F_{\wt Y})$: take any such triple and replace the embedding $e$ with the composition
\[ 
	e' \colon \wt X \xra{i} W_2(X/\wt S) \ra W_2(Y/\wt S) \xra{t(F_{\wt Y})} \wt Y.
\]

Let  $(\wt Y, e, F_{\wt Y})$ be a retracting triple.  By the universal property of the PD-envelope, we see that the triangle in the diagram below commutes
\[ 
	\xymatrix{
		\wt X\ar[d]_i \ar[r]^{\bar e} & \bar Y \ar[r] & \wt Y \\ 
		W_2(X/\wt S) \ar[rr] \ar@{.>}[ur]^{\bar t} &  & W_2(Y/\wt S) \ar[u]^{t(F_{\wt Y})},
	}
\]
which is precisely the assertion that the right cell in \eqref{eqn:small-triangle} commutes. This finishes the proof of Theorem~\ref{thm:fhodge1}.

\begin{remark}
The name `retracting' is justified by the case $\wt X = \wt Y$, $e={\rm id}$, in which case it simply states that $t(F_{\wt X})\colon W_2(X/\wt S) \to \wt X$ is a retraction of the closed immersion $i$. If $\wt X$ admits a Frobenius lifting, it also admits a retracting one, and consequently retracting Frobenius liftings exist locally on $\wt X$. 

If $S=\Spec k$, the space of retracting Frobenius liftings is either empty or a torsor under $\Hom(\Omega^1_X, B^1_X)$. For example, if $X$ is an abelian variety and $\sigma$ is the unique $F$-splitting, then $\wt X(\sigma)$ admits multiple Frobenius liftings, but exactly one of them is retracting. This is the Serre--Tate canonical lifting.
\end{remark}

\subsection{Comparison with de Rham--Witt theory}
\label{ss:fhodge1-drw}

Let $S=\Spec k$. We finish the discussion of Theorem~\ref{thm:fhodge1} by comparing the map $t\colon \smash{\Omega^\bullet_{\wt X/\wt S}}\to W_2 \cO_X$ \eqref{eqn:def-t} constructed in the proof with the map $t'$ obtained by composing the canonical quasi-isomorphism $\Omega^\bullet_{\wt X/\wt S} \isom Ru_* \cO_{X/\wt S}$ with crystalline cohomology, the quasi-isomorphism $Ru_* \cO_{X/\wt S} \isom W_2\Omega^\bullet_X$ coming from de Rham--Witt theory, and the projection $W_2\Omega^\bullet_X \to W_2\cO_X$.

To this end, let $(\wt Y, e, F_{\wt Y})$ be a triple as in \S\ref{ss:fhodge1pf-cris}. By definition \cite[p.\ 600]{Illusie_deRhamWitt}, the isomorphism $\Omega^\bullet_{\wt X/\wt S} \isom Ru_* \cO_{X/\wt S} \isom W_2\Omega^\bullet_X$ equals the composition
\[ 
	\Omega^\bullet_{\wt X/\wt S} \isom \Omega^\bullet_{\wt Y/\wt S}|_{\bar Y} \xra{t(F_{\wt Y})} \Omega^\bullet_{W_2 X} \ra W_2 \Omega^\bullet_X.
\]
After projecting to the $0$-th term in each complex, we see that $t=t'$. Combining this with the commutativity of \eqref{eqn:fhodge2} (a priori for a retracting triple), we obtain:

\begin{cor} \label{cor:f1drw}
	The following diagram commutes in the derived category
	\[
		\xymatrix{
			& W_2\Omega^\bullet_X \ar[r] & W_2\cO_X \ar[dd]^{i^*} \\
			Ru_* \cO_{X/\wt S} \ar[dr]_{\rotatebox{-35}{$\sim$}} \ar[ur]^{\rotatebox{35}{$\sim$}} & & \\
			& \Omega^\bullet_{\wt X/\wt S} \ar[r] & \cO_{\wt X}.
		} 	
	\]
\end{cor}

\subsection{Uniqueness of the canonical lifting}
\label{ss:fhodge1-uniq}

Using Ogus' results on Griffiths transversality in crystalline cohomology \cite{OgusGT}, we can show that for $1$-ordinary varieties with trivial canonical class in odd characteristic, the property expressed in Theorem~\ref{thm:fhodge1} characterizes the canonical lifting.

\begin{thm} \label{thm:canlift-uniq}
	Suppose that $p>2$. Let $X/k$ be a $1$-ordinary variety with trivial canonical class of dimension $d$, and let $\wt X/W_2(k)$ be the canonical lifting of $X'$, i.e.\ $\wt X = \wt X(\sigma)$ for the unique $F$-splitting $\sigma$ on $X$. Suppose that for every lifting of $X$ to $W_2(k)$, the Hodge groups are free and the Hodge spectral sequence degenerates. 
	Then $\wt X$ is the unique lifting of $X$ over $W_2(k)$ for which the crystalline Frobenius $\phi$ preserves $F^1 H^d_{\rm dR}(\wt X/W_2(k))$.
\end{thm}

\begin{proof}
Since the Frobenius is bijective on $H^d(X, W_2\cO_X)$, we obtain a decomposition
\[ 
	H^d_{\rm cris}(X/W_2(k)) \isom H^d(X, W_2\cO_X) \oplus H^d(X, W_2\Omega^{\bullet\geq 1}_X)
\]
where the Frobenius is divisible by $p$ on the second summand. By Corollary~\ref{cor:f1drw}, we see that $F^1 H^d_{\rm dR}(\wt X/W_2(k))$ coincides with the second summand. Consequently, $F^1 H^d_{\rm dR}(\wt X/W_2(k))$ is the unique Frobenius-stable lifting of $F^1 H^d_{\rm dR}(X/k)$. 

It remains to show that different liftings of $X$ give rise to different liftings of $F^1 H^d_{\rm dR}(X/k)$. More precisely, if $\bar X/W_2(k)$ is another lifting such that the image $\bar F^1$ of $F^1 H^d_{\rm dR}(\bar X/W_2(k))$ under the crystalline isomorphism 
\[ 
	H^d_{\rm dR}(\bar X/W_2(k)) \isomlong H^d_{\rm dR}(\wt X/W_2(k))
\]
equals $\wt F^1 = F^1 H^d_{\rm dR}(\wt X/W_2(k))$, then $\bar X \isom \wt X$. The obstruction to $\bar F^1 = \wt F^1$ is the map $\bar F^1 \to H^d_{\rm dR}(\wt X/W_2(k))/\wt F^1$, which since $\bar F^1 = \wt F^1$ mod $p$ vanishes mod $p$ and hence factors through a map $F^1 H^d_{\rm dR}(X/k)\to H^d_{\rm dR}(X/k) / F^1$. Since $p>2$, by Griffiths transversality \cite{OgusGT} this map vanishes on $F^2$, and hence factors through a map 
\[ 
	\eta(\bar X) \colon H^{d-1}(X, \Omega^1_X) \isom \gr^1 H^d_{\rm dR}(X/k) \ra \gr^0 H^d_{\rm dR}(X/k) \isom H^d(X, \cO_X).
\]

Varying $\bar X\in \Def_X(W_2(k))$, writing $v = \bar X - \wt X \in H^1(X, T_X)$ we obtain a map
\[ 
	v\mapsto \eta(\wt X + v) \colon H^1(X, T_X) \ra \Hom(H^{d-1}(X, \Omega^1_X), H^d(X, \cO_X)).
\]
By \cite[Corollary~2.12]{OgusGT}, this map coincides with map obtained by cup product, which is an isomorphism in our case. 
\end{proof}

\begin{cor} \label{cor:comparison-st1}
	In the situation of Theorem~\ref{thm:canlift-uniq}, if $X$ is either an abelian variety or a K3 surface, then the canonical lifting $\wt X$ agrees modulo $p^2$ with the canonical lifting constructed in \cite{DeligneIllusieKatz}.
\end{cor}

\begin{remark} Extending the arguments in \cite[\S 2]{OgusGT}, one can show that if $X_0/k$ is a proper smooth variety  for which the map
	\[
		H^1(X_0, T_{X_0}) \ra \Hom(H^{*-k}(X_0, \Omega^k_{X_0}), H^{*-k+1}(X_0, \Omega^{k-1}_{X_0}))		
	\]
	is injective, and if $X$ and $X'$ are two liftings of $X_0$ to $W_n(k)$ satisfying (DEG) and with free Hodge groups, then $X\isom X{}'$ if and only if the crystalline isomorphism 
	\[
		H^*_{\rm dR}(X/W_n(k))\isom H^*_{\rm dR}(X{}'/W_n(k))
	\] 
	identifies $F^k H^*_{\rm dR}(X/W_n(k))$ with $F^k H^*_{\rm dR}(X{}'/W_n(k))$.

	Specializing to the case $X_0$ a $1$-ordinary variety with trivial canonical class of dimension $d$ satisfying (NCT) and (DEG), $*=d$ and $k=1$, we see that there exists exactly one lifting of $F^1 H^d_{\rm dR}(X_0/k)$ to $H^d_{\rm cris}(X_0/W_n(k))$ preserved by Frobenius (arguing as in the proof above), and that this lifting comes from at most one lifting of $X_0$ over $W_n(k)$. It is unclear however whether such a lifting of $X_0$ exists if $n>2$, unless $X_0$ is an abelian variety or a K3 surface.
\end{remark}

\subsection{The first higher Hasse--Witt operation}
\label{ss:fhodge1-hw1}

Let $(X, \sigma)$ be an $F$-split smooth projective scheme over $k$, and let $\wt X = \wt X(\sigma)$ be the canonical lifting of $X'$ to $W_2(k)$. Suppose that $p>2$, that 
the Hodge groups of $\wt X$ are free and that its Hodge spectral sequence degenerates. 
Let $H = H^n_{\rm dR}(\wt X/W_2(k))$ for some $n\geq 0$ and let 
\[ 
	\phi \colon H\ra H
\]  
be the crystalline Frobenius. By the easy case of the divisibility estimates \cite{Mazur}, $\phi$ maps $F^1$ into $pH$ and vanishes on $F^2$ (as $p>2$). Moreover, by Theorem~\ref{thm:fhodge1}, we have $\phi(F^1)\subseteq F^1$, so that $\phi(F^1)\subseteq pH\cap F^1 = pF^1$ since $F^1$ is a direct summand of $H$. There is therefore a unique morphism
\[ 
	\beta = \frac{\phi}{p} \colon H^{n-1}(X, \Omega^1_{X}) \to H^{n-1}(X, \Omega^1_{X})
\]
such that the following diagram commutes
\[ 
	\xymatrix{
		H \ar[d]_\phi & F^1 \ar@{_{(}->}[l] \ar[d]_\phi \ar@{->>}[r] & F^1\otimes k \ar@{->>}[r] \ar[d] & (F^1/F^2)\otimes k \ar[r]^\sim \ar[d] & H^{n-1}(X, \Omega^1_X) \ar[d]^\beta \\
		H & F^1 \ar@{_{(}->}[l] & F^1\otimes k \ar@{_{(}->}[l]^-{\times p} \ar@{->>}[r] & (F^1/F^2)\otimes k \ar[r]^\sim  & H^{n-1}(X, \Omega^1_X).
	}
\]
In addition, the Hasse--Witt operation $\HW(0)\colon H^n(X, \cO_X)\to H^n(X, \cO_X)$ is bijective by Proposition~\ref{prop:1-ord}, so that the first higher Hasse--Witt operation  \eqref{eqn:hwi}
\[
	\HW(1)\colon H^{n-1}(X, \Omega^1_{X}) \ra H^{n-1}(X, \Omega^1_{X})
\]
is defined.

\begin{prop} \label{prop:beta-hw1}
	We have $\beta = \HW(1)$. 
\end{prop}

\begin{proof}
Let $x\in F^1$, and let $y$ be its image in $(F^1/F^2)\otimes k = H^{n-1}(X, \Omega^1_X)$. The element $\beta(y)$ is characterized by
\[ 
	\phi(x) = p\cdot \beta(y) \text{ mod }F^2.
\]
Let $z\in F_1/F_0$ be the image of $y$ under the isomorphism \eqref{eqn:hodge-conj-gr}
\[ 
	C^{-1} \colon H^{n-1}(X, \Omega^1_X) \isomlong F_1/F_0.
\]
By \cite[8.26.3]{BerthelotOgus}, we have $z = \phi(x)/p\text{ mod }F_0$. The element $HW(1)(y)$ is constructed as follows: there exists a unique $t\in F_1\cap (F^1\otimes k)$ lifting $z$, and $HW(1)(y)$ is the image of $t$ in $(F^1/F^2)\otimes k$. Since $\phi$ preserves $F^1$, we must have $\phi(x) = pt$, so that $\beta(y) = t = HW(1)(y)$ modulo $F^2$.   
\end{proof}


\section{Modular Frobenius liftings}
\label{s:modular-flift}

\noindent
In this section we show how the construction of the canonical lifting produces certain `modular' liftings of Frobenius modulo $p^2$. For motivation, suppose that there exists a~fine moduli space $\wt M$ over $W_2(k)$ parametrizing certain $1$-ordinary varieties with trivial canonical class, and that the morphism $\wt M\to \Spec W_2(k)$ is flat. Let $\wt X/\wt M$ be the universal family, and let $X/M$ be its reduction modulo $p$. The construction of the relative canonical lifting yields a family $\wt X{}'/\wt M$ lifting $X' = F_{M}^* X$. Since $M$ is a fine moduli space, there exists a unique morphism $\wt F\colon \wt M\to \wt M$ such that $\wt X{}' \isom \wt F{}^* \wt X$. This morphism is the desired lifting of Frobenius on the moduli space $\wt M$.

There are two ways of avoiding the problem of the non-existence of a fine moduli space. First, one can consider a suitable moduli stack $\mathscr{M}$. Second, one can try to construct the Frobenius lifting $\wt F$ on the base of a~family $X/M$ which is no longer universal, but which is universal formally locally at every point. We find the second approach more useful, and we deal with it first, coming back to stacks in \S\ref{ss:stacks}.

Before proceeding, let us recall the following result, which implies that the data of the modular Frobenius lifting is the same as the assignment of a canonical lifting $\wt X/\wt S$ of $X'$ to every family $X/S$. 

\begin{prop}[{\cite[Proposition~3.5.3]{AchingerWitaszekZdanowicz}}] \label{prop:awzmagic}
	Let $\wt M$ be a flat $W_2(k)$-scheme endowed with a Frobenius lifting $F_{\wt M}\colon \wt M\to \wt M$, and let $M=\wt M\otimes k$. Then for every flat $W_2(k)$-scheme $\wt S$ endowed with a map $f\colon S\to M$ there exists a canonically defined morphism $g\colon \wt S\to \wt M$ lifting $F_M\circ f$:
	\[ 
	\xymatrix{
		\wt S\ar@{.>}[r]_{\tilde f} \ar@{.>}@/^2em/[rr]^g & \wt M \ar[r]^{F_{\wt M}} & \wt M \\
		S\ar[u] \ar[r]_f & M\ar[r]_{F_M} \ar[u] & M. \ar[u]
	}
	\]
	This lifting $g$ is functorial, and it commutes with every Frobenius lifting on $\wt S$. If $\tilde f$ is a~lifting of $f$, then $g = F_{\wt M} \circ \tilde f$.
\end{prop}

The morphism $g$ is defined as the composition
\[ 
	\wt S \xra{\theta} W_2(S) \xra{W_2(f)} W_2(M) \xra{t(\wt F)} \wt M  
\]
where on functions $\theta(x_0, x_1) = \tilde x_0^p + p\tilde x_1$ for arbitrary liftings $\tilde x_0, \tilde x_1 \in \cO_{\wt S}$, and where $t(\wt F)$ is the Cartier map, defined by $t(\wt F)(y) = (y, \delta(y))$ where $\wt F(y) = y^p + p\delta(y)$. This defines a mapping
\[ 
	\left\{\text{liftings of Frobenius on }\wt M\right\}
	\ra
	\left\{\text{functorial associations }f\mapsto g\text{ as in Proposition~\ref{prop:awzmagic}}\vphantom{\wt M}\right\}
\]
which is a bijection, with inverse given by evaluation at ${\rm id}_M$, at least if $M$ is smooth over $k$ (so that $\tilde f$ exists locally on $\wt S$).

\subsection{Frobenius liftings in modular families}
\label{ss:froblift-modular}

Let $\wt M$ be a flat $W_2(k)$-scheme locally of finite type, equipped with a flat family $\wt X/\wt M$ whose fibers are $1$-ordinary varieties with trivial canonical class. Suppose that this family is `formally universal' in the following sense: for every $\bar m\in M(\bar k)$, the natural transformation
\[ 
	\Spf \hat \cO_{\wt M\otimes W_2(\bar k), \bar m} \ra \Def_{X_{\bar m}/W_2(\bar k)}
\]
induced by the base change of $\wt X/\wt M$ to $\Spf \hat \cO_{\wt M\otimes W_2(\bar k)}$, is an isomorphism. Let $\wt X{}'/\wt M$ be the canonical lifting of $X'$ (Corollary~\ref{cor:can-lift-cy}).

\begin{prop}
	There exists a unique lifting of Frobenius $F_{\wt M}\colon \wt M\to \wt M$ such that the families $F_{\wt M}^* \wt X/\wt M$ and $\wt X{}'/\wt M$ are locally isomorphic.
\end{prop}

\begin{proof}
The uniqueness is clear, as any two such lifts have to coincide on $\Spf \hat \cO_{\wt M\otimes W_2(\bar k), \bar m}$ for every $\bar m\in M(\bar k)$. This implies that the existence of $F_{\wt M}$ can be checked locally. We let 
\[ 
	\wt P = {\rm Isom}_{\wt M\times \wt M}(\pi_1^* \wt X, \pi_2^* \wt X{}');
\] 
this is the scheme representing the functor associating to $S/W_2(k)$ the set of triples 
\[
	(\alpha\colon S\to \wt M, \beta\colon S\to \wt M, \iota\colon \alpha^* \wt X{}'\isom \beta^* \wt X).
\]
Being an open subscheme of the Hilbert scheme of the product family $\pi_1^* \wt X\times \pi_2^* \wt X{}'$, it is a scheme locally of finite type over $W_2(k)$. We let $\pi\colon \wt P\to \wt M$ be the projection mapping $(\alpha, \beta, \iota)$ to $\alpha$. 

The problem of constructing $F_{\wt M}$ can now be restated as follows: the data $({\rm id}_M, F_M, {\rm id})$ produces a section $\sigma\colon M\to P$ of $\pi\colon P\to M$, and we wish to extend it to a section $\wt M\to \wt P$ locally on $M$. For this, it is enough to show that $\pi$ is smooth along the image of $\sigma$. 

To this end, we can replace the base $\wt M$ with $\Spf \hat \cO_{\wt M\otimes W_2(\bar k), \bar m}$. By pro-representability, there exists a unique $\beta\colon \Spf \hat \cO_{\wt M\otimes W_2(\bar k), \bar m} \to \Spf \hat \cO_{\wt M\otimes W_2(\bar k), \bar m}$ such that $\wt X{}' \isom \beta^* \wt X$. Fix such an isomorphism $\iota_0$, then the base change of $\wt P$ to $\Spf \hat \cO_{\wt M\otimes W_2(\bar k), \bar m}$ becomes identified with the automorphism scheme of $\wt X{}'$, which is smooth by \cite[Theorem~2.6.1]{Sernesi}.
\end{proof}

\begin{remark}
The lifting of Frobenius $F_{\wt M}\colon \wt M\to \wt M$ corresponds to a morphism 
\[
	t(F_{\wt M})\colon W_2(M)\to \wt M,
\]
and hence yields a natural extension of $X/M$ to $W_2(M)$. See \cite{BorgerGuerney} for similar considerations.
\end{remark}

\subsection{Frobenius liftings on moduli stacks}\label{ss:stacks}

Consider the stack $\mathscr{M}$ on the big \'etale site of $\ZZ/p^2\ZZ$ associating to a scheme $S$ over $\ZZ/p^2\ZZ$ the groupoid of flat schemes $X/S$ whose geometric fibers are $1$-ordinary varieties with trivial canonical class. We wish for $\mathscr{M}$ to be a flat (or even smooth) Deligne--Mumford stack, but this is false already for K3 surfaces. In any case, the diagonal of $\mathscr{M}$ is representable by schemes locally of finite type.

To circumvent this difficulty, we choose an open substack $\mathscr{U}\subseteq \mathscr{M}$ which is a Deligne--Mumford stack and flat over $\ZZ/p^2\ZZ$. In simple terms, there exists a flat algebraic space $S$ over $\ZZ/p^2\ZZ$ and a surjective \'etale map $f\colon S\to \mathscr{U}$. The map $f$ corresponds to a family $X/S \in \mathscr{M}(S)$ which is formally universal.

We note that there exists a largest such substack $\mathscr{U}$: take $S_0$ to be the disjoint union of all flat algebraic spaces over $\ZZ/p^2\ZZ$ endowed with an \'etale map to $\mathscr{M}$, and let $S_1 = S_0 \times_\mathscr{M} S_0$ (which is an algebraic space, since the diagonal of $\mathscr{M}$ is representable). Then 
\[
\xymatrix{
	S_1 \ar@<-.5ex>[rr] \ar@<.5ex>[rr] & & S_0 
}
\]
defines a groupoid presentation of an open substack $\mathscr{U}\subseteq \mathscr{M}$. A $1$-ordinary variety with trivial canonical class $X/k$ defines a point in $\mathscr{U}(k)$ if e.g.\  $H^2(X, \cO_X) = 0$ and its formal deformation functor is pro-representable and flat over $W_2(k)$.

The Frobenius liftings constructed in \S\ref{ss:froblift-modular} on $S_1$ and $S_0$ descend to produce a lifting of Frobenius $F_{\mathscr{U}}\colon \mathscr{U}\to \mathscr{U}$.

\begin{remark}
	Since abelian varieties and K3 surfaces possess non-algebraic deformations, the above construction does not apply on the nose to the moduli stacks of such varieties. However, it can be adapted to produce Frobenius liftings on the moduli of principally polarized abelian varieties: a principal polarization is an isomorphism $A\to A^\vee$ and it is clear that such a structure is preserved by the construction of the canonical lifting. Similarly, for ordinary K3 surfaces in characteristic $p>2$, the formal subscheme of $\Def_X$ where a given line bundle $L$ deforms is preserved by the Frobenius lifting by \cite[2.2.2]{DeligneIllusieKatz}, and hence one can construct a Frobenius lifting modulo $p^2$ on the moduli stack of polarized ordinary K3 surfaces in odd characteristic.
\end{remark}


\section{Frobenius and the Hodge filtration (II)}
\label{s:fhodge2}

\subsection{The modular Frobenius lifting preserves $F^1$}
\label{ss:fhodge2}

Let $X_0/k$ be a $1$-ordinary variety with trivial canonical class. Suppose that the deformation functor $\wt S=\Def_{X/W_2(k)}$ is pro-representable and smooth over $W_2(k)$. 
Corollary~\ref{cor:can-lift-cy} applied to the reduction $X/S$ of the universal family $\wt X/\wt S$ yields a family $\wt X{}'$ lifting $X' = F_S^* X$, which is isomorphic to $F_{\wt S}^* \wt X$ for a unique Frobenius lifting $F_{\wt S}\colon S\to S$.

\begin{thm} \label{thm:fhodge2}
	The crystalline Frobenius map induced by the Frobenius lifting $F_{\wt S}$
	\[ 
		\phi(F_{\wt S})\colon  H^d_{\rm dR}(\wt X/\wt S) \ra H^d_{\rm dR}(\wt X/\wt S)
	\]
	maps $F^1  H^d_{\rm dR}(\wt X/\wt S)$ into itself.  
\end{thm}

\noindent
This should be contrasted with the fact that for the canonical lift of a family $X/S$, \emph{every} Frobenius lifting on $\wt S$ preserves $F^1$ (Theorem~\ref{thm:fhodge1}). Interestingly, we deduce Theorem~\ref{thm:fhodge2} from this fact.

\begin{proof}
We have to show that the map of $\cO_S$-modules
\[ 
	F_{\wt S}^* F^1 \hookrightarrow F_{\wt S}^* H^d_{\rm dR}(\wt X/\wt S) \xra{\phi(F_{\wt S})} H^d_{\rm dR}(\wt X/\wt S) \to H^d_{\rm dR}(\wt X/\wt S)/F^1
\]
vanishes. Since $F_{\wt S}\colon \wt S\to \wt S$ is faithfully flat, it is enough to check this after pull-back by $F_{\wt S}$. On the other hand, $F_{\wt S}^* \wt X$ is by definition of $F_{\wt S}$ the canonical lifting of the family $X/S$. By Theorem~\ref{thm:fhodge1}, we know that 
\[ 
	\phi(F_{\wt S}) \colon H^d_{\rm dR}(F_{\wt S}^* \wt X/\wt S) \ra H^d_{\rm dR}(F_{\wt S}^* \wt X/\wt S)
\]
preserves $F^1$. But $H^d_{\rm dR}(F_{\wt S}^* \wt X/\wt S) \isom F_{\wt S}^* H^d_{\rm dR}(\wt X/\wt S)$ compatibly with Frobenius, and hence the result. 
\end{proof}

Our remaining goal in this section will be to employ the above result in order to relate the Frobenius lifting $F_{\wt S}$ to the first higher Hasse--Witt operation of $X$ (\S\ref{ss:fhodge2-hw1}). This will be then applied in \S\ref{ss:fhodge2-app} to show that $X$ can be deformed to a $2$-ordinary variety, and to show that the property of preserving $F^1$ actually characterizes $F_{\wt S}$. This in turn allows us to compare $F_{\wt S}$ with the classical construction in the case of abelian varieties and K3 surfaces.

\subsection{The modular Frobenius lifting and the first higher Hasse--Witt operation}
\label{ss:fhodge2-hw1}

Let $\wt S = \Spf W_2(k)[[t_1, \ldots, t_r]]$ and let $F_{\wt S}\colon \wt S\to \wt S$ be a Frobenius lifting. Let $H = (H, \nabla, F^\bullet, \phi)$ be a Hodge $F$-crystal over $\wt S$ (Definition~\ref{def:hodgefcrys}). Let 
\[
	\psi \colon T_{\wt S} \ra \Hom(F^1/F^2, H/F^1)
\]
be the map induced by $\nabla$. Suppose that the Frobenius map
\[ 
	\phi(F_{\wt S})F_{\wt S}^* \colon H\ra H
\]
maps $F^1$ into $F^1$ and vanishes on $F^2$, yielding maps
\[ 
	\phi_0 \colon H/F^1 \ra H/F^1 \quad\text{and}\quad \phi_1 \colon F^1/F^2 \ra F^1/F^2.
\] 
Suppose in addition that the map $\phi_0$ is bijective.

\begin{prop} \label{prop:xi-hw1}
	In the above situation, the following square commutes
	\[
		\xymatrix{
			T_{\wt S} \ar[rr]^-\psi \ar[d]_-{d F_{\wt S}} & & \Hom(F^1/F^2, H/F^1) \ar[d]^\gamma \\
			F_{\wt S}^* T_{\wt S} \ar[rr]_-{F^*_{\wt S} \psi}  & & \Hom(F_{\wt S}^* (F^1/F^2), F_{\wt S}^* (H/F^1)),
		}
	\]
	where $\gamma$ maps a homomorphism $f\colon F^1/F^2 \to H/F^1$ to the composition
	\[ 
		F_{\wt S}^* (F^1/F^2) \xra{\phi_1} F^1/F^2 \xra{f} H/F^1 \xra{\phi_0^{-1}} F^*_{\wt S} (H/F^1).
	\]
\end{prop}  

\begin{proof}
Let $v\in T_A$, we show that the two images of $v$ in $\Hom(F_{\wt S}^* (F^1/F^2), F_{\wt S}^* (H/F^1))$ agree after post-composition with $\phi_0\colon F_{\wt S}^* (H/F^1)\to H/F^1$. Let $x\in F^1$, then 
\[ 
	\phi_0(\gamma(\psi(v))(F_{\wt S}^*(x)\text{ mod }F^2)) = (\nabla_v \phi(F_{\wt S})(x)) \text{ mod }F^1
\]
while
\[ 
	\phi_0(F_{\wt S}^*(\psi)(dF_{\wt S}(v))F_{\wt S}^*(x)\text{ mod }F^2)) = \phi(F_{\wt S})(F_{\wt S}^* \nabla_{dF_{\wt S}(v)} x)\text{ mod }F^1.
\]
Thus the assertion follows from the fact that $\phi(F_{\wt S})$ is horizontal, i.e.
\[ 
	\nabla_v (\phi(F_{\wt S})(F_{\wt S}^* x)) = \phi(F_{\wt S})(\nabla_v F_{\wt S}^* x) = \phi(F_{\wt S})(F_{\wt S}^* \nabla_{dF_{\wt S}(v)} x). \qedhere 
\]
\end{proof}

\begin{cor} \label{cor:xi-hw1} 
	Let $X$ be a $1$-ordinary variety with trivial canonical class with unobstructed deformations over $W_2(k)$. Suppose that $p>2$ and that the assumptions in \S\ref{ss:hodgefcrystal} are satisfied. Then the following square commutes
	\[ 
		\xymatrix{
			H^1(X, T_X) \ar[r]^-\sim \ar[d]_{\xi\, \eqref{eqn:xi0}} & \Hom(H^{d-1}(X, \Omega^1_X), H^d(X, \cO_X)) \ar[d]^{\Hom(\HW(1), \HW(0)^{-1})} \\
			H^1(X, T_X) \ar[r]_-{\sim} & \Hom(H^{d-1}(X, \Omega^1_X), H^d(X, \cO_X)). 
		}
	\]
	Consequently, $X$ is $2$-ordinary if and only if the Frobenius lifting $F_{\wt S}$ is ordinary (Definition~\ref{def:flift-ordinary}).
\end{cor}

\subsection{Applications}
\label{ss:fhodge2-app}

If $\wt S$ is a smooth scheme over $W_2(k)$ with a Frobenius lift $F_{\wt S}$, then the induced map $\xi = \frac{1}{p}F_{\wt S}^* \colon F_S^* \Omega^1_{S'/k}\to \Omega^1_{S/k}$ is injective and hence generically an isomorphism. We deduce that the $2$-ordinary locus is a dense open subset in the moduli of $1$-ordinary varieties as in Corollary~\ref{cor:xi-hw1}. More formally, we have:

\begin{cor} \label{cor:2-ordinary-def} 
	Every $1$-ordinary variety $X$ as in Corollary~\ref{cor:xi-hw1} admits a formal deformation over $k[[t]]$ whose generic fiber is $2$-ordinary.
\end{cor}

\begin{cor} \label{cor:uniq-froblift-moduler} 
	Under the assumptions of Corollary~\ref{cor:xi-hw1}. The Frobenius lifting $F_{\wt S}$ in Theorem~\ref{thm:fhodge2} is the unique Frobenius lifting on $\wt S$ for which $\phi(F_{\wt S})$ preserves $F^1 H^d_{dR}(\wt X/\wt S)$.
\end{cor}

\begin{proof}
Suppose that $F$ is another Frobenius lifting preserving $F^1$ and write 
\[ 
	\wt S = \Spf W_2(k)[[t_1, \ldots, t_r]], 
	\quad F_{\wt S}(t_i) = t_i^p + pf_i,
	\quad F(t_i) = t_i^p + pf'_i. 
\] 
We have to prove that $f_i=f'_i$ for all $i$. Proposition~\ref{prop:xi-hw1} implies that $dF_{\wt S} = dF$, so that $f_i - f'_i = g_i^p$ for some $g_i\in k[[t_1, \ldots, t_r]]$. Let $v = \sum g_i \frac{\partial}{\partial t_i}$; if $F_{\wt S}\neq F$, then $v\neq 0$. Since the `Kodaira--Spencer' map
\[ 
	\gr\nabla \colon T_{\wt S} \ra \Hom(F^1/F^2, H/F^1)
\]
is an isomorphism, there exists an $x\in F^1$ such that $\nabla_{v} x \notin F^1$.

As $p>2$, the `change of Frobenius' formula \eqref{eqn:changefrob} gives
\begin{align}
	\phi(F_{\wt S})F_{\wt S}^* x &=
	\phi(F)F^* x 
	+ \sum_{i=1}^r p(f_i - f'_i) \phi(F)F^*(\nabla_{\frac{\partial}{\partial t_i}} x) \nonumber \\
	&= \phi(F)F^* x +  p\phi(F_S)F_S^* (\nabla_v x)
	\quad
	\text{for } x\in H.  \label{eqn:change-frob}
\end{align}
Since the Frobenius of $S$ induces an isomorphism on $H/F^1\otimes k$, we have 
\[ 
	\phi(F)F^*(\nabla_{v} x) \notin F^1
\]
so the corresponding term in \eqref{eqn:change-frob} is not in $F^1$, a contradiction.
\end{proof}

\noindent
Applying \cite[Appendix]{DeligneIllusieKatz}, we obtain:

\begin{cor} \label{cor:comparison-st2}
	In the situation of Theorem~\ref{thm:fhodge2}, if $p>2$ and $X$ is either an abelian variety or a K3 surface, then the Frobenius lifting $F_{\wt S}$ on $\wt S = \Def_{X/W_2(k)}$ agrees with the restriction modulo $p^2$ of the Serre--Tate Frobenius lifting constructed in \cite{DeligneIllusieKatz}.
\end{cor}


\section{Canonical coordinates}
\label{s:cancoord}

\subsection{Ordinary liftings of Frobenius}

Let $\wt S = \Spf W_2(k)[[t_1, \ldots, t_r]]$, and let $S$ be its reduction mod $p$. 

\begin{lemma} \label{lemma:teich}
	Let $\wt F\colon \wt S\to \wt S$ be a lifting of Frobenius. There exists a unique map over $W_2(k)$
	\[ 
		f \colon \Spf W_2(k) \ra \wt S
	\]
	which commutes with the Frobenius lifts.
\end{lemma}

\begin{proof}
This is standard. See e.g.\ \cite[\S 1.1]{KatzDwork}.
\end{proof}

\begin{defin}
	We call $f$ the \emph{Teichm\"uller point} associated to $\wt F$, and denote the ideal of its image by $J_{\wt F}\subseteq \cO_{\wt S}$.
\end{defin}

Since the map $\wt F{}^*\colon \wt F{}^* \Omega^1_{\wt S}\to \Omega^1_{\wt S}$ vanishes modulo $p$, there exists a unique mapping
\[ 
	\xi\colon F_S^* \Omega^1_{S} \ra \Omega^1_{S}
\]
such that $\wt F{}^*(\omega) = p \cdot \xi(\omega\,{\rm mod}\, p)$, as in \eqref{eqn:def-xi}.
Explicitly, if $\wt F(t_i) = t_i^p + pf_i$ then
\[ 
	\xi(dt_i) = t_i^{p-1} dt_i + df_i.
\]

\begin{defin}[{\cite[III \S 1]{Mochizuki}}]\label{def:flift-ordinary}
	We call a lifting of Frobenius $\wt F\colon \wt S\to \wt S$ \emph{ordinary} if the map $\xi$ is an isomorphism.
\end{defin}

\begin{remarks} \label{rmk:ord-froblift}
\begin{enumerate}[1.]
	\item The determinant of $\xi$ sends the generator 
	\[
		\wt F{}^*(dt_1 \wedge \ldots \wedge dt_r)
	\]
	to 
	\[ 
		(t_1^{p-1} dt_1 + df_1)\wedge \ldots \wedge (t_r^{p-1} dt_r + df_r),
	\]
	which equals $df_1\wedge\ldots\wedge df_r$ modulo $(t_1, \ldots, t_r)\cdot \Omega^1_{S}$. In particular, the lifting $\wt F$ is ordinary if and only if the Jacobian
	\[ 
		\det \left[ \frac{\partial f_i}{\partial t_j} \right] \in \cO_{S}
	\]
	has invertible constant term.
	\item If $\wt F$ is ordinary, the fixed point set of $\xi$
	\[ 
		\{ \omega \in \Omega^1_{S}\,|\, \xi(\omega)=\omega\}
	\]
	is an $\FF_p$-vector space of dimension $r$, and any basis yields a basis of $\Omega^1_{S}$ as an $\cO_S$-module. 
	\item An easy calculation using the fact that $\wt F{}^*$ vanishes on $\Omega^2_{\wt S}$ shows that $d\circ \xi$ vanishes. In particular, forms fixed by $\xi$ are always closed. 
\end{enumerate}
\end{remarks}

\subsection{Comparison with the $p$-th power map} 

We are interested in identifying the group  $Q(\wt F)$ consisting of $q\in \cO^\times_{\wt S}$ for which $\wt F{}^*(q) = q^p$. For this, we note first that `subtracting' the diagrams with exact rows
\[ 
	\xymatrix{
		0\ar[r] & \cO_S \ar[r]^{1+pf} \ar[d]_{F^*_X} & \cO_{\wt S}^\times \ar[r] \ar[d]^{\wt F^*} & \cO_S^\times \ar[r] \ar[d]^{F_X^*} & 1\\
		0 \ar[r] & \cO_S \ar[r]^{1+pf} & \cO_{\wt S}^\times \ar[r] & \cO_S^\times \ar[r] & 1
	}
	\quad\text{and}\quad
	\xymatrix{
		0\ar[r] & \cO_S \ar[r]^{1+pf} \ar[d]_{0} & \cO_{\wt S}^\times \ar[r] \ar[d]^{f\mapsto f^p} & \cO_S^\times \ar[r] \ar[d]^{f\mapsto f^p} & 1\\
		0 \ar[r] & \cO_S \ar[r]^{1+pf} & \cO_{\wt S}^\times \ar[r] & \cO_S^\times \ar[r] & 1
	}
\]
yields
\[ 
	\xymatrix{
		0\ar[r] & \cO_S \ar[d]_{F_S^*} \ar[r] & \cO_{\wt S}^\times \ar[r] \ar[d]|-{f\mapsto \wt F{}^*(f)/f^p } & \cO_S^\times \ar[r] \ar[d]^{1} & 1\\
		0 \ar[r] & \cO_S \ar[r] & \cO_{\wt S}^\times \ar[r] & \cO_S^\times \ar[r] & 1.
	}
\]
Snake lemma then gives an exact sequence
\[ 
	0\ra Q(\wt F) \ra \cO_S^\times \xra{\delta} \cO_S / \cO_S^{p} .
\]
The composition of $\delta$ with the injective map $d \colon \cO_S / \cO_S^{p}\to Z\Omega^1_S$ maps $f$ to $\xi(df)f^{-p} - d\log f$, which vanishes precisely if $\xi(d\log f) = d\log f$. This gives
\[
	Q(\wt F) \isomto \{ f\in \cO_S^\times \, | \, \xi(d\log f) = d\log f\}.
\]

\begin{lemma} \label{lemma:cancoord}
	Suppose that $\omega\in \Omega^1_S$ satisfies $\xi(\omega) = \omega$. Then $\omega = d\log q$ for some $q\in \cO_S^\times$, uniquely defined up to $\cO_S^{\times p}$.
\end{lemma}

\begin{proof}
Since $d\omega = 0$ (Remark~\ref{rmk:ord-froblift}.3), we have $C(\omega) = C(\xi(\omega)) = \omega$. On the other hand, we have the short exact sequence \cite[Corollaire~2.1.18, p.\ 217]{Illusie_deRhamWitt}
\[ 
	1 \ra \cO_S^\times / \cO_S^{\times p}\ra Z\Omega^1_S \xra{C - 1} \Omega^1_S \ra 0.
\]
which yields the result.
\end{proof}

\subsection{Canonical multiplicative coordinates}

We will call a tuple of elements $\tilde q_1, \ldots, \tilde q_r \in \cO_{\wt S}$ \emph{multiplicative coordinates} if 
\[ 
	(\tilde q_1 -1, \ldots, \tilde q_r -1 ,p )  
\]
is the maximal ideal of $W_2(k)[[t_1, \ldots, t_r]]$. 

Multiplicative coordinates give rise to the Frobenius lifting $\wt F$ defined by $\wt F{}^*(\tilde q_i) = \tilde q_i^p$ which is easily seen to be ordinary. For the converse, we have the following result, whose version `over $W(k)$' was proved in \cite[Corollaire~1.4.5]{DeligneIllusieKatz} under the assumption $p>2$.

\begin{prop} \label{prop:cancoord}
	Let $\wt F\colon A\to A$ be an ordinary lifting of Frobenius and let $\omega_1, \ldots, \omega_r$ be a basis of $\{ \omega\in \Omega^1_S\,|\,\xi(\omega) = \omega\} \cong \FF_p^r$. 
	\begin{enumerate}[(a)]
		\item There exist multiplicative coordinates $q_1, \ldots, q_r \in \cO_S$, unique up to $\cO_S^{\times p}$, such that $\omega_i = d\log q_i$. They admit unique liftings to multiplicative coordinates $\tilde q_i \in \cO_{\wt S}$ such that $\wt F{}^*(\tilde q_i) = \tilde q_i^p$. 
		\item The ideal $J_{\wt F}$ is generated by $\tilde q_i - 1$.
		\item If $q'_1, \ldots, q'_r \in \cO_S$ have the same property, then 
	\[ 
		\tilde q_i - \tilde q'_i \in \wt F(J_{\wt F}) = (\tilde q_1^p -1, \ldots, \tilde q_r^p - 1)
	\]
	\end{enumerate}
\end{prop}

\begin{proof}
The first assertion follows directly from Lemma~\ref{lemma:cancoord}. For the second, by Lemma~\ref{lemma:teich} it is enough to note that the map $W_2(k)[[\tilde q_1 - 1, \ldots, \tilde q_r-1]]\to W_2(k)$ sending all $q_i$ to $1$ commutes with the Frobenius lift. For the last assertion, again by Lemma~\ref{lemma:cancoord} we have $q'_i = u^p_i q_i$ for some $1$-units $u_i \in \cO_S$. Let $\tilde u_i \in \cO_{\wt S}$ be a lifting such that $\tilde u_i - 1$ belongs to $(q_1 -1, \ldots, q_r - 1) = J_{\wt F}$. Then $\tilde q'_i = \tilde u_i^p \tilde q_i$, so
\[ 
	\tilde q'_i - \tilde q_i = \tilde q_i (\tilde u_i^p - 1) = \tilde q_i \wt F{}^*(\tilde u_i - 1) \in \wt F(J_{\wt F}). \qedhere
\]
\end{proof}

\section{Applications to isotriviality}
\label{s:isotriv}
 
\noindent
In this section, we deduce certain `hyperbolicity' properties of the $2$-ordinary locus of the moduli of $d$-dimensional varieties with trivial canonical class satisfying the conditions 
\begin{equation}\label{eqn:conditions_applications}
	H^1(X, \cO_X) = H^2(X, \cO_X) = H^{d}(X, \Omega^1_X) = 0.
\end{equation}
More precisely, suppose that $S$ is a simply connected base with no non-zero global one-forms. We prove that a family of $2$-ordinary varieties with trivial canonical class, smooth deformation space, and satisfying the above conditions and conditions (NCT) and (DEG) over $S$ is necessarily isotrivial.  We remark that results in a similar direction appear in \cite[Theorem~3.4]{OgusFGT}.  The essential point of the argument is that the moduli space classifying varieties in question is a smooth Deligne--Mumford substack of $\mathscr{U}$ (defined in \S\ref{ss:stacks}) equipped with an ordinary Frobenius lifting (see Definition~\ref{def:flift-ordinary}).  We denote this stack by $\mathscr{U}^{\rm sm}$.  

First, we state a general lemma which we shall then use for the above stack.

\begin{lemma}\label{lem:morphism_dm_stack_no_forms}
Let $S$ be a smooth simply connected variety such that $H^0(S,\Omega^1_S) = 0$.  Suppose that $\mathscr{M}$ is a Deligne--Mumford stack admitting an isomorphism $F^*\Omega^1_{\mathscr{M}} \isomto \Omega^1_{\mathscr{M}}$.  Then for every morphism $f \colon S \to \mathscr{M}$ the differential $df \colon f^*\Omega^1_{\mathscr{M}} \to \Omega^1_{S}$ is zero.   
\end{lemma}
\begin{proof}
Let $f \colon S \to \mathscr{M}$ be a morphism.  The pullback of the given isomorphism induces an isomorphism $F_S^*f^*\Omega^1_{\mathscr{M}} \to f^*\Omega^1_{\mathscr{M}}$.  By the results of \cite{Lange_Stuhler} this implies that $f^*\Omega^1_{\mathscr{M}}$ is \'etale trivializable, and hence trivial because $S$ is simply connected.  Consequently, the differential $df \colon \cO_S^{n} \isom f^*\Omega^1_{\mathscr{M}} \to \Omega^1_S$ is induced by a collection of sections $\Omega^1_S$, and therefore is zero.  
\end{proof}

\begin{remark}
In characteristic zero the assumption $H^0(S,\Omega^1_S) = 0$ follows from $\pi_1^{\et}(S) = 0$.  This is no longer true in characteristic $p$, for example for supersingular Enriques surfaces in characteristic $2$ \cite[Proposition~III~7.3.8, p.\ 658]{Illusie_deRhamWitt}.
\end{remark}

We now apply the above lemma for the stack $\mathscr{U}^{\rm sm}$.  By \S\ref{ss:stacks}, we know that $\mathscr{U}^{\rm sm}$ is a smooth Deligne--Mumford stack, and admits an isomorphism $F^*\Omega^1_{\mathscr{U}^{\rm sm}} \to \Omega^1_{\mathscr{U}^{\rm sm}}$ induced by the differential of the modular Frobenius lifting.

\begin{prop} \label{prop:isotriv1}
Let $S$ be a smooth simply connected variety over a perfect field $k$ of characteristic $p>2$ such that $H^0(S,\Omega^1_S) = 0$.  Then there are no non-isotrivial families over $S$ of $2$-ordinary varieties with trivial canonical class satisfying conditions (NCT), (DEG), and \eqref{eqn:conditions_applications}.
\end{prop}
\begin{proof}
Suppose $X \to S$ is a non-isotrivial family of $2$-ordinary Calabi--Yau varieties satisfying conditions \eqref{eqn:conditions_applications}, and let $f \colon S \to \mathscr{U}^{\rm sm}$ be the corresponding non-constant  morphism.  The dual of the differential $df \colon f^*\Omega^1_{\mathscr{U}^{\rm sm}} \to \Omega^1_S$ can be identified with the Kodaira--Spencer class ${\rm KS}_{X/S}$, and therefore we may apply the following reasoning repeatedly.  Either, $df$ is non-zero and we obtain a contradiction with Lemma~\ref{lem:morphism_dm_stack_no_forms} or $df$ is zero, which by the mentioned identification, allows us to apply to Theorem~\ref{thm:ks_zero_F_descent} to descend $X \to S$ via the relative Frobenius $F_{S/k} \colon S \to S'$, that is, construct a factorization $f = f' \circ F_{S/k}$ with $f' \colon S' \to \mathscr{U}^{\rm sm}$ non-constant.  We now substitute $f$ with $f'$ and reiterate the argument.  By Lemma~\ref{lem:frobenius_degree} given below, the second situation can only happen finitely many times and therefore the proof is finished.
\end{proof}

\begin{lemma}\label{lem:frobenius_degree}
Let $f \colon \mathscr{X} \to \mathscr{Y}$ be a non-constant morphism (i.e., it does not factor through $\Spec(k)$ on any \'etale chart) of noetherian Deligne--Mumford stacks, with $\mathscr{X}$ connected.  Then there exists an integer $n \geq 0$ such that $f$ does not factorize via the $n$-th power of the Frobenius morphism $F^n_{\mathscr{X}/k} \colon \mathscr{X} \to \mathscr{X}^{(n)}$
\end{lemma}
\begin{proof}
Let $X \to \mathscr{X}$ and $Y \to \mathscr{Y}$ be the \'etale charts such that $f$ lifts to a morphism of schemes $g$ in the diagram 
\[
	\xymatrix{
		X \ar[r]^{g}\ar[d] & Y \ar[d] \\
		\mathscr{X} \ar[r]_{f} & \mathscr{Y}.
	}
\]
Assume that $f$ factors through the $n$-th power of the Frobenius morphism $F^n\colon \mathscr{X} \to \mathscr{X}^{(n)}$.  By \cite[XIV=XV \S{}1 $n^\circ$2, Pr. 2(c)]{SGA5} this implies that $g$ factors through the $n$-th power of the Frobenius of $X$.  For noetherian schemes, the claim is then clear using Krull's intersection theorem.
\end{proof}

\noindent The following remark was suggested to us by Yohan Brunebarbe.

\begin{remark}
It is conjectured that on simply connected varieties all $F$-isocrystals are trivial (see \cite{EsnaultShiho} for a proof in a special case), and hence the Newton polygon does not change among the fibers.  Consequently, in a family of varieties with trivial class satisfying assumption (NCT) -- necessary for Hodge polygons to be constant, if one geometric fiber is $2$-ordinary then all the fibers are.  The conjecture and our results therefore imply that no simply connected variety intersects the $2$-ordinary locus of the moduli space of varieties with trivial canonical class.
\end{remark}

\subsection{Alternative version}

In what follows we give a more direct proof of the above results without any reference to the moduli stack.  Let $S$ be an integral scheme over $k$, and let $f \colon X \to S$ be a family of smooth varieties of dimension $d$, such that the Hodge sheaves
\[
R^{d-i}f_*\Omega^{i}_{X/S} \quad 0 \leq i \leq d
\] are locally free, and the relative Hodge to de Rham spectral sequence
\[
R^jf_*\Omega^{d-i}_{X/S} \quad \Longrightarrow \quad H^{d}_{\rm dR}(X/S)
\] degenerates.  Using \cite[\S 4]{DeligneIllusie}, we see these conditions are satisfied if the family lifts mod $p^2$ and is of relative dimension $d < p$, or if all the fibers satisfy assumptions (NCT) and (DEG).

Using \cite[\S 2.3]{KatzAlgSoln}, the conditions imply that the Hasse--Witt operators can be defined in the relative setting.  In order to define $\HW(0)_{X/S}$, we let $f' \colon X' \to S$ be the Frobenius twist of $X/S$.  The base change map induces a morphism $F_S^*R^df_*\cO_X \to R^df'_*\cO_{X'}$ whose composition with the natural map $R^df'_*\cO_{X'} \to R^df'_*F_{X/S,*}\cO_X = R^df_*\cO_X$ gives a~morphism 
\[
{\rm HW}(0)_{X/S} \colon F_S^*R^df_*\cO_X \to R^df_*\cO_X,
\]
whose specialization for every $s \in S$ is identified with the $p$-linear map 
\[
{\rm HW}(0) \colon H^d(X_s,\cO_{X_s}) \to H^d(X_s,\cO_{X_s}) \quad \text{(cf.\ \S\ref{ss:hasse_witt})}.
\] 
Once ${\rm HW}(i-1)_{X/S}$ is an isomorphism, we can define 
\[
{\rm HW}(i)_{X/S} \colon F_S^*R^{d-i}f_*\Omega^i_{X/S} \to R^{d-i}f_*\Omega^i_{X/S},
\]
which specializes to ${\rm HW}(i)$ on the fibers.  

\begin{prop} \label{prop:isotriv2}
Suppose $S$ is a smooth simply connected variety over a perfect field $k$ of characteristic $p>0$ such that $H^0(S, \Omega^1_S) = 0$. Then there are no non-isotrivial families over $S$ of $2$-ordinary varieties with trivial canonical class satisfying (NCT) and (DEG).  
\end{prop}

\begin{proof}
	First, we observe that the conditions required to define relative Hasse--Witt operators are satisfied.  Since the fibers are $2$-ordinary, the maps ${\rm HW}(0)_{X/S}$ and ${\rm HW}(1)_{X/S}$ are well-defined and yield isomorphisms
	\[
		F_S^*(R^df_*\cO_X) \isomto R^df_*\cO_X 
		\quad 
		F_S^*(R^{d-1}f_*\Omega^1_{X/S}) \isomto R^{d-1}f_*\Omega^1_{X/S}.
	\]
	By \cite{Lange_Stuhler}, this implies that $R^df_*\cO_X \isom f_*\omega_{X/S}$ and $R^{d-1}f_*\Omega^1_{X/S}$ are trivial vector bundles.  Consequently, using the fact that the fibers have trivial canonical class and Corollary~\ref{cor:rfomega}, we see that the dual 
	\[
		(R^{d-1}f_*\Omega^1_{X/S})^\vee \isom R^1f_*(T_{X/S} \otimes \omega_{X/S}) \isom R^1f_*T_{X/S}.
	\]
	is also trivial.  As in the above proof, using the Frobenius descent (see Theorem~\ref{thm:ks_zero_F_descent}), we may assume that Kodaira--Spencer map $T_{S} \to  R^1f_*T_{X/S}$ is non-zero.  Its dual therefore furnishes a non-zero differential $1$-form on $S$ which finishes the proof.  
\end{proof}


\section{Serre--Tate theory a'la Nygaard}
\label{s:nygaard}

\noindent
Let $X/k$ be a $2$-ordinary variety with trivial canonical class defined over an \emph{algebraically closed} field $k$. We set
\[ 
	d = \dim X, \quad r = \dim H^{d-1}(X, \Omega^1_X) = \dim H^1(X, T_X).
\] 
In addition to (NCT) and (DEG), we assume that
\[
	H^{d-1}(X, \cO_{X}) = 0 = H^d(X, \Omega^1_{X}),
\]
(equivalently, $H^1(X, \cO_X) = 0 = H^0(X, T_X)$, i.e.\ ${\rm Pic}(X)$ and ${\rm Aut}(X)$ are reduced and discrete). The first assumption implies that the Artin--Mazur formal group $\Phi_{X}$ is pro-representable (see \S\ref{ss:amfg}). The second one implies that $\Def_X$ is pro-representable, and follows from the first if Hodge symmetry holds for $X$. 

The goal of this section is to show that Nygaard's construction of Serre--Tate theory for K3 surfaces \cite{Nygaard} (see \cite{WardThesis} for the Calabi--Yau threefold case) works without much change for varieties $X$ as above under the assumption that the deformation space is smooth. We keep the discussion quite brief; in particular, we avoid any mention of enlarged Artin--Mazur groups, though they are still `morally' present in the background (see Remark~\ref{rmk:p-div-gp}).

\subsection{Hodge--Witt and flat cohomology}

The following two $\ZZ_p$-modules will be of key importance:
\[ 
	U = H^d_{\rm cris}(X/W(k))^{\phi=1}, \quad E = H^d_{\rm cris}(X/W(k))^{\phi=p}.
\]
We shall relate $U$ and $E$ to the Hodge--Witt cohomology $H^d(X, W\cO_X)$ and $H^{d-1}(X, W\Omega^1_X)$ and the cohomology groups $H^d_\et(X, \ZZ_p)$ and $H^d_{\rm fl}(X, \ZZ_p(1))$. Recall (e.g.\ from \cite[II 5.1]{Illusie_deRhamWitt}) that the latter group is by definition the inverse limit of the flat cohomology groups $H^d_{\rm fl}(X, \mu_{p^n})$, which by the $p^n$-Kummer sequence can be identified with the Zariski cohomology groups $H^{d-1}(X, \cO^\times_X/\cO^{\times p^n}_X)$, and that the maps
\[
	d\log [-]\colon \cO^\times_X/\cO^{\times p^n}_X \ra W_n\Omega^1_X
\]
define by passing to the limit a map $d\log\colon H^d_{\rm fl}(X, \ZZ_p(1)) \to H^{d-1}(X, W\Omega^1_X)$.

\begin{lemma} \label{lemma:nygaard-pre}
	The $\ZZ_p$-modules $H^d_\et(X, \ZZ_p)$ and $H^d_{\rm fl}(X, \ZZ_p(1))$ are torsion free, the maps
	\[ 
		H^d_\et(X, \ZZ_p)\ra H^d(X, W\cO_X)^{F=1}, \quad H^d_{\rm fl}(X, \ZZ_p(1)) \xra{d\log} H^{d-1}(X, W\Omega^1_X)^{F=1}
	\]
	are bijective and induce isomorphisms
	\[ 
		H^d_\et(X, \ZZ_p)\otimes W(k) \isomto H^d(X, W\cO_X),
		\quad
		H^d_{\rm fl}(X, \ZZ_p(1))\otimes W(k)\isomto H^{d-1}(X, W\Omega^1_X).
	\]
\end{lemma}

\begin{proof}
A careful analysis of the proof of \cite[Lemma~7.1]{Bloch_Kato} shows that the cohomology vanishings given by Proposition~\ref{prop:1-ord}(ii) and Lemma~\ref{lemma:2-ord} are sufficient to obtain the required assertions. 
\end{proof}

The $2$-ordinarity assumption implies by Proposition~\ref{prop:m-ord} that we have a decomposition 
\begin{equation} \label{eqn:decomp-cris1} 
	(H^d_{\rm cris}(X/W(k)), \phi) \isom (U\otimes W(k), 1\otimes \sigma) \oplus (E\otimes W(k), p\otimes \sigma) \oplus (H_{\geq 2}, p^2 \phi').
\end{equation}
On the other hand, under certain soon to be verified assumptions \cite[Cor.\ 4.4 p.\ 201]{IllusieRaynaud}, the slope and conjugate de Rham--Witt spectral sequences yield a decomposition
\begin{equation} \label{eqn:decomp-cris2}
		H^d_{\rm cris}(X/W(k) \isom H^d(X, W\cO_X) \oplus H^{d-1}(X, W\Omega^1_X) \oplus H^d(X, \tau_{\geq 2}W\Omega^\bullet_X)
\end{equation}

\begin{prop} \label{prop:nygaard-pre}
	The decompositions \eqref{eqn:decomp-cris1} and \eqref{eqn:decomp-cris2} coincide. Consequently,
	\[ 
		U \isom H^d(X, W\cO_X)^{F=1} \isom H^d_\et(X, \ZZ_p)
	\]
	is a free $\ZZ_p$-module of rank one and
	\[
		E \isom H^{d-1}(X, W\Omega^1_X)^{F=1} \isom H^d_{\rm fl}(X, \ZZ_p(1)).
	\]
	is a free $\ZZ_p$-module of rank $r$.
\end{prop}

\begin{proof}
To apply \cite[Cor.\ 4.4 p.\ 201]{IllusieRaynaud} we need to check that $H^{d-1}(X, ZW\Omega^1_X)$ and $H^d(X, ZW\Omega^1_X)$ are finitely generated $W$-modules and that the inclusion induces an isomorphism
$
	H^{d-1}(X, ZW\Omega^1_X) \isomlong H^{d-1}(X, W\Omega^1_X).
$
By \cite[(1.3.1), p.\ 174]{IllusieRaynaud}, we have
\[
	H^{*}(X, ZW\Omega^1_X) \isomto \varprojlim_F H^*(X, W\Omega^1_X).	
\]
The required assertions follow now from the fact that $F$ is bijective on the finitely generated $W$-modules $H^{d-1}(X, W\Omega^1_X)$ and $H^d(X, W\Omega^1_X)(=0)$. Since a~decomposition of the type \eqref{eqn:decomp-cris1} is unique, we obtain the desired assertion.
\end{proof}

We note for future reference that the above assertions imply in particular that the composition 
\[ 
	E \isom H^d_{\rm fl}(X, \ZZ_p(1)) \ra H^d_{\rm fl}(X, \mu_p) \isom H^{d-1}_\et(X, \cO_X^\times/\cO_X^{\times p}) \xra{d\log} H^{d-1}(X, \Omega^1_X) 
\]
induces an isomorphism $E\otimes_{\ZZ_p} k \isomto H^{d-1}(X, \Omega^1_X)$.


\subsection{Serre--Tate theory for $2$-ordinary varieties}

Recall that by \S\ref{ss:amfg} we have an isomorphism $\Phi_X \isomto U\otimes \hat\GG_m$ inducing the identity $H^d(X, \cO_X)=H^d(X, \cO_X)$ on the tangent spaces. By rigidity, for every deformation $\wt X$ of $X$ over an Artinian local ring $A$ with residue field $k$, there exists a unique isomorphism $\Phi_{\wt X}\isomto U\otimes \hat \GG_{m,A}$ extending the given one. 

Let $\wt X/A$ be as above, and let $n\geq 0$ be such that $\mathfrak{m}_A^n = 0$. We have a short exact sequence
\begin{equation} \label{eqn:unitsmodpn}
	0\ra 1+\mathfrak{m}_A\cO_{\wt X} \ra \cO^\times_{\wt X}/\cO^{\times p^n}_{\wt X} \ra \cO^\times_{X}/\cO^{\times p^n}_{X} \ra 0.
\end{equation}
We let $\beta_n$ be the composition
\[
	\beta_n \colon H^d_{\rm fl}(X, \mu_{p^n}) \isom H^{d-1}(X, \cO^\times_{X}/\cO^{\times p^n}_{X}) \xra{\delta} H^d(X, 1+\mathfrak{m}_A\cO_{\wt X}) = \Phi_{\wt X}(A) = U \otimes (1+\mathfrak{m}_A).
\]
where $\delta$ is the connecting homomorphism induced by \eqref{eqn:unitsmodpn}. The sequence of maps $\beta_n$ for $n\gg 0$ induces a homomorphism on the inverse limit
\[ 
	\beta\colon E = H^d_{\rm fl}(X, \ZZ_p(1)) = \varprojlim H^d_{\rm fl}(X, \mu_{p^n}) \ra \Phi_{\wt X}(A) = U \otimes (1+\mathfrak{m}_A).
\]
Varying $A$ and $\wt X$, we obtain a natural transformation
\[ 
	\gamma\colon \Def_X \ra T,
\]
where $T = \Hom(E, U)\otimes \hat\GG_m$ is the formal torus with co-character group $\Hom(E, U)$.

\begin{prop}
	The map $\gamma$ induces an isomorphism on tangent spaces
	\[ 
		\gamma\colon \Def_X(k[\varepsilon]/(\varepsilon^2)) \isomlong T(k[\varepsilon]/(\varepsilon^2)).
	\]
\end{prop}

\begin{proof}
We note first that using Proposition~\ref{prop:nygaard-pre} and Lemma~\ref{lemma:nygaard-pre} the target can be identified with
\begin{align} 
	 \Hom_{\ZZ_p}(H^d_{\rm fl}(X, \ZZ_p(1)), H^d(X, 1+\varepsilon \cO_X))  
	 &= \Hom_{\FF_p}(H^d_{\rm fl}(X, \mu_p), H^d(X, \cO_X)) \nonumber \\
	 &= \Hom_{\FF_p}(H^{d-1}_\et(X, \cO_X^\times/\cO_X^{\times p})), H^d(X, \cO_X)) \label{eqn:identif-a} \\
	 &= \Hom_k(H^{d-1}(X, \Omega^1_X), H^d(X, \cO_X)) \label{eqn:identif-b}.
\end{align}
On the other hand, a deformation $\wt X\in \Def_X(k[\varepsilon]/(\varepsilon^2))$ corresponds to an element ${\rm KS}_{\wt X}\in H^1(X, T_X)$. Cup product with ${\rm KS}$ induces a homomorphism
\[ 
	\cup {\rm KS}\colon H^{d-1}(X, \Omega^1_X) \ra H^d(X, \cO_X),
\]
which coincides with the boundary homomorphism coming from the short exact sequence
\[ 
	0\ra \cO_X \xra{d\varepsilon} \Omega^1_{\wt X} \ra \Omega^1_X \ra 0.
\]
This defines a bijection
\begin{equation} \label{eqn:ksisom} 
	\Def_X(k[\varepsilon]/(\varepsilon^2)) \isom H^1(X, T_X) \isomlong \Hom_k(H^{d-1}(X, \Omega^1_X), H^d(X, \cO_X)).
\end{equation}
We shall prove that this map equals $\gamma$ under our identification \eqref{eqn:identif-b}.

The short exact sequence above fits inside a commutative diagram with exact rows
\[ 
	\xymatrix{
		0\ar[r] & 1+\varepsilon \cO_{X} \ar[d]_\log \ar[r] & \cO_{\wt X}^\times/\cO_{\wt X}^{\times p} \ar[d]^{d\log} \ar[r] & \cO_X^\times/\cO_X^{\times p} \ar[d]^{d\log} \ar[r] & 0\\
		0\ar[r] & \varepsilon\cO_{X} \ar[r] & \Omega^1_{\wt X/k} \ar[r] & \Omega^1_{X/k} \ar[r] & 0,
	}
\]
which induces a commutative square
\[ 
	\xymatrix{
		H^{d-1}(X, \cO_X^\times/\cO_X^{\times p}) \ar[r]^\delta \ar[d]_{d\log} & H^d(X, 1+\varepsilon \cO_X) \ar[d]^{\log}_{\rotatebox{90}{$\sim$}} \\
		H^{d-1}(X, \Omega^1_X) \ar[r]_{\cup {\rm KS}_{\wt X}} & H^d(X, \cO_X).
	}
\]
Since $\delta$ above corresponds to $\gamma(\wt X)$ under the identification \eqref{eqn:identif-a}, it follows that ${\rm KS}_{\wt X}$ corresponds to $\gamma(\wt X)$ under the identification \eqref{eqn:identif-b}, as desired.
\end{proof}

\begin{cor} \label{cor:nygaard}
	If $X$ has unobstructed deformations over $W_m(k)$ for some $m\geq 1$, then $\gamma$ induces an isomorphism 
	\begin{equation} \label{eqn:defgamma}
		\gamma\colon \Def_{X/W_m(k)} \isomlong T \otimes W_m(k).
	\end{equation}
\end{cor}

\begin{remark} \label{rmk:p-div-gp}
In the above discussion, we chose to be as direct as possible. In particular, the reader might be surprised by the absence of $p$-divisible groups. The enlarged formal group $\Psi_X$ generalizing the one used by Nygaard can be reconstructed as follows. Let $\wt X/A$ be deformation of $X$, corresponding to a homomorphism 
\[
	\beta\colon H^d_{\rm fl}(X, \ZZ_p(1))\ra \Phi_{\wt X}(A) = H^d(\wt X, 1+ \mathfrak{m}_A \cO_{\wt X}).
\]
One can construct a morphism $\bar\beta$ making the diagram below commute.
\[
	\xymatrix{
		0\ar[r] & H^d_{\rm fl}(X, \ZZ_p(1)) \ar[r] \ar[d]_\beta & H^d_{\rm fl}(X, \ZZ_p(1))\otimes \QQ_p \ar[r] \ar[d]_{\bar\beta} & H^d_{\rm fl}(X, \ZZ_p(1)) \otimes \QQ_p/\ZZ_p \ar[r] \ar[d] & 0 \\
		0\ar[r] & H^d(\wt X, 1+ \mathfrak{m}_A \cO_{\wt X}) \ar[r] & H^d_{\rm fl}(\wt X, \mu_{p^\infty}) \ar[r] & H^d_{\rm fl}(X, \mu_{p^\infty}) \ar[r] & 0. 
	}
\]
By a reasoning analogous to \cite[Corollary~1.4]{Nygaard}, the right arrow is an isomorphism, thus setting $\Psi_{\wt X}(B) = H^d_{\rm fl}(\wt X\otimes B, \mu_{p^\infty})$ as in \cite[Corollary~1.5]{Nygaard}, the bottom row can be interpreted as the extension of $p$-divisible groups 
\[ 
	0\ra \Phi_{\wt X} \ra \Psi_{\wt X} \ra H^d_{\rm fl}(X, \ZZ_p(1))\otimes \QQ_p/\ZZ_p \ra 0,
\]
which is the pushout of the top row above under $\beta$. Consequently, this extension corresponds to $\beta$ under Messing's isomorphism 
\[ 
	\Hom_{\ZZ_p}(H^d_{\rm fl}(X, \ZZ_p(1)), \Phi_{\wt X})\isomlong \Ext^1_A(H^d_{\rm fl}(X, \ZZ_p(1))\otimes \QQ_p/\ZZ_p, \Phi_{\wt X}),
\]
and our torus $T$ is identified with the deformation space of $\Psi_X$.
\end{remark}

\begin{question} \label{question:polarized}
Let $L$ be a line bundle on $X$. The assumption that $H^{d-1}(X, \cO_X)=0$ is equivalent to $H^1(X, \cO_X)=0$, so the forgetful transformation $\Def_{X, L}\to \Def_X \isom T$ is a~closed immersion. Is its image preserved by the Frobenius lifting on $\Def_X$, e.g.\ the kernel of a character of $T$? 

This is the case for K3 surfaces \cite[2.2.2]{DeligneIllusieKatz}: the crystalline Chern class $c_1(L)$ lies in $E = H^2(X/W(k))^{F=p}$. If $\chi\in \Hom(T, \hat\GG_m)= \Hom(U, E)$ maps a basis element of $U\isom \ZZ_p$ to $c_1(L)$, then $\Def_{X, L} = \ker \chi$.

A positive answer would provide the moduli spaces of polarized $2$-ordinary varieties with trivial canonical class with a natural Frobenius lifting.
\end{question}

\subsection{Comparison with canonical liftings constructed in \S\ref{s:canlift}}
\label{ss:nygaardp2}

Corollary~\ref{cor:nygaard} gives in particular a notion of a canonical lifting: we denote by $X_m/W_{m+1}(k)$ the lifting of $X$ corresponding to the neutral element of $T(W_{m+1}(k))$ under $\gamma$ \eqref{eqn:defgamma}. 

\begin{prop}
	The canonical lifting $\wt X/W_2(k)$ constructed in \S\ref{s:canlift} is isomorphic to $X'_1$. In particular, the crystalline Frobenius preserves $F^1 H^d_{\rm dR}(X_1/W_2(k))$.
\end{prop}

\begin{proof}
Recall from Lemma~\ref{lemma:teichmueller} that the restriction map $\cO_{\wt X}^\times \to \cO_{X'}^\times$ admits a section. Consequently, the sequences \eqref{eqn:unitsmodpn} are split, and hence $\beta=0$. 
\end{proof}

\begin{question} \label{question:crysf1}
Does the crystalline Frobenius preserve $F^1 H^d_{\rm dR}(X_m/W_{m+1}(k))$ for $m> 1$? Does the Frobenius lifting on $\Def_X$ corresponding to the $p$-th power map on $T$ under \eqref{eqn:defgamma} satisfy the natural analog of Theorem~\ref{thm:fhodge2}?
\end{question}

For abelian varieties and K3 surfaces \cite{DeligneIllusieKatz}, the answer is to both questions is positive. For more general varieties with trivial canonical class, one runs into the problem that not every lifting of $F^1 H^d_{\rm dR}(X/k)$ to $H^d_{\rm cris}(X/W_m(k))$ comes from a lifting of $X$. We were able to overcome this difficulty (an avatar of Griffiths transversality constraints on period maps) only modulo $p^2$.


\appendix

\section{Finite height}
\label{app:yobuko}

\noindent
In \cite{Yobuko}, Yobuko defines the notion of a quasi-$F$-splitting, more general than an $F$-splitting, and proves that a smooth quasi-$F$-split variety can be lifted modulo $p^2$, generalizing the argument for smooth $F$-split varieties \cite[\S 8.5]{IllusieFrobenius}. He also shows that for Calabi--Yau varieties, being quasi-$F$-split is equivalent to the height of the associated Artin--Mazur formal group being finite. In this section, we give a somewhat different point of view on Yobuko's construction. Based on this, we extend our construction of a~canonical lifting associated to an $F$-splitting to quasi-$F$-split varieties, thus showing that there is a preferred lifting mod $p^2$. In particular, we show that the smoothness assumption used by Yobuko is not necessary. Some of our results were observed earlier by Adrian Langer (unpublished).

It would be interesting to extend this construction to families and to generalize the Serre--Tate theory discussed in this paper to varieties with trivial canonical class of finite height in the spirit of \cite{NygaardOgus}. 

\subsection{The canonical lifting}

Yobuko defines a \emph{quasi-$F$-splitting of level $m$} on an \mbox{$\FF_p$-scheme} $X$ as an additive map $\sigma\colon W_m\cO_X\to \cO_X$ satisfying $\sigma(1)=1$ and which is $F$-linear in the sense that
\[ 
	\sigma(Fx\cdot y) = x_0 \cdot \sigma(y).
\]
We have a short exact sequence
\[ 
	0\ra W_m\cO_X \xra{V} W_{m+1}\cO_X \xra{R^m} \cO_X \ra 0.
\]

\begin{lemma} \label{lemma:ideal}
	Let $\sigma\colon W_m\cO_X\to \cO_X$ be a quasi-$F$-splitting of level $m$ on an $\FF_p$-scheme $X$. Then the image of $\ker(\sigma)$ under $V$ is an ideal in $W_{m+1}\cO_X$.
\end{lemma}

\begin{proof}
Suppose that $x\in W_m\cO_X$ satisfies $\sigma(x)=0$, and let $y\in W_{m+1}\cO_X$. Then
\[
	y\cdot V(x) = V(F(y)\cdot x)
\] 
and
\[ 
	\sigma(F(y)\cdot x) = y\cdot \sigma(x) = 0.\qedhere
\]
\end{proof}

\begin{cor} \label{cor:lifting}
	Let $\sigma\colon W_m\cO_X\to \cO_X$ be a quasi-$F$-splitting of level $m$ on an \mbox{$\FF_p$-scheme} $X$. Then the quotient
	\[
		\cO_{\wt X} = W_{m+1}\cO_X / V(\ker \sigma)
	\]
	defines a lifting of $X$ over $\ZZ/p^2\ZZ$, fitting inside a pushout diagram of exact sequences
	\[ 
		\xymatrix{
			0\ar[r] & W_m\cO_X \ar[d]_\sigma \ar[r]^V & W_{m+1}\cO_X \ar[d]^\pi \ar[r]^{R^m} & \cO_X \ar@{=}[d] \ar[r] & 0 \\
			0\ar[r] & \cO_X \ar[r] & \cO_{\wt X}\ar[r] & \cO_X \ar[r] & 0.
		}
	\]
\end{cor}

\begin{proof}
Since $\sigma$ is surjective, the left square is a pushout. To prove that $\cO_{\wt X}$ is a lifting, we need to check that the composition
\[ 
	\cO_{\wt X} \ra \cO_X \ra \cO_{\wt X}
\]
equals multiplication by $p$. Let $\tilde f\in \cO_{\wt X}$ be the image of
\[ 
	f= (f_0,\ldots, f_m) \in W_{m+1}\cO_X.
\]
Then the image of $\tilde f$ in $\cO_X$ is $f_0$, which is the image under $\sigma$ of
\[ 
	Ff = (f_0^p, \ldots, f_m^p) \in W_m\cO_X.
\]
This in turn has image $V(F f) = pf$ in $W_m\cO_X$, which maps to $p\tilde f$ under $\pi$.
\end{proof}

\begin{defin}
	We call the lifting $\wt X$ constructed in Corollary~\ref{cor:lifting} the \emph{canonical lifting} associated to the quasi-$F$-splitting $\sigma$.
\end{defin}

\subsection{The sheaf $\mathcal{F}_m \cO_X$ and Witt vectors mod $p$}

Suppose now that $X$ is a smooth scheme over a perfect field $k$ of characteristic $p$. We denote by 
\[ 
	C\colon Z\Omega^1_X \ra \Omega^1_X
\]
the Cartier operator. The sheaf $Z_m\Omega^1_X\subseteq F^m_* \Omega^1_X$ consists of local sections $\omega$ such that 
\[ 
	\omega, C(\omega), \ldots C^{m-1}(\omega)
\]
are closed. The subsheaf $B_m\Omega^1_X\subseteq Z_m\Omega^1_X$ consists of local sections $\omega$ such that $C^{m-1}(\omega)\in B^1_X$. The map 
\[ 
	D_m \colon F_* W_m \cO_X \ra B_m\Omega^1_X,
	\quad
	(f_0, \ldots, f_{m-1}) \mapsto f_0^{p^{m-1}-1}df_0 +\ldots +df_{m-1}
\]
is $W_m\cO_X$-linear, and there is a short exact sequence of $W_m\cO_X$-modules
\[ 
	0\ra W_m\cO_X  \xra{F}  F_*W_m\cO_X \xra{D_m} B_m\Omega^1_X\ra 0. 
\] In this situation, Yobuko defined a natural extension of $\cO_X$-modules (denoted $(e_m)$)
\[
	0\ra \cO_X\ra \mathcal{F}_m\cO_X\ra B_m\Omega^1_X \ra 0
\] 
and showed that $\cO_X$-linear splittings of this sequence correspond to quasi-$F$-splittings as defined previously. The exact sequence $(e_m)$ is defined by the pullback diagram
\[ 
	\xymatrix{
		0\ar[r] & \cO_X \ar@{=}[d] \ar[r] & \mathcal{F}_m\cO_X \ar[r]\ar[d] & B_m\Omega^1_X \ar[d]^{C^{m-1}} \ar[r] & 0 \\
		0 \ar[r] & \cO_X \ar[r] & F_* \cO_X \ar[r] & B_1\Omega^1_X \ar[r] & 0.
	}
\] 
In particular, if $X$ is $F$-split, the bottom row is, and hence so is $(e_m)$. Yobuko observes that $(e_m)$ can also be defined as a pushout
\[ 
	\xymatrix{
		0\ar[r] & W_m\cO_X \ar[d]_{R^{m-1}} \ar[r]^-F & F_* W_m\cO_X \ar[d]\ar[r]^-{D_m} & B_m\Omega^1_X \ar@{=}[d] \ar[r] & 0 \\
		0 \ar[r] & \cO_X \ar[r] & \mathcal{F}_m \cO_X \ar[r] & B_m \Omega^1_X \ar[r] & 0.
	}
\]
It follows that quasi-$F$-splittings are precisely $\cO_X$-linear splittings of $(e_m)$. 


We shall now elucidate the sheaf $\mathcal{F}_m$ and its $\cO_X$-algebra structure. Let $X$ be an $\FF_p$-scheme. Following \cite[\S 3]{HesselholtMadsen}, we denote by
\[ 
	\overline{W}_m \cO_X = W_m \cO_X / p\cdot W_m\cO_X
\]
the mod-$p$ reduction of the ring of Witt vectors of length $m$ over $\cO_X$. Recall from \emph{op.cit.}\ that the $p$-th power of the Teichm\"uller map
\[ 
	[-]^p \colon \cO_X \ra W_m\cO_X
\]
is additive modulo $p$, and therefore induces an $\cO_X$-algebra structure 
\[ 
	\rho = [-]^p \, {\rm mod}\, p\colon \cO_X \ra \overline{W}_m\cO_X
\]
on $W_m\cO_X$. In addition, the following triangle commutes
\[
	\xymatrix{
		\cO_X \ar[r]^-\rho \ar[dr]_F & \overline{W}_m \cO_X \ar[d]^{R^{m-1}} \\ 
		& \cO_X.
	}
\]
Consequently, $\rho$ is injective if $X$ is reduced. We denote its cokernel by $B_m$, so that there is a short exact sequence of $\cO_X$-modules
\[ 
	0\ra \cO_X\ra \overline{W}_m\cO_X \ra B_m\ra 0.
\]
The formula
\[ 
	V(x)\cdot V(y) = V(FV(x)\cdot V(y)) = p\cdot V(xy)
\]
implies that the multiplication in $\overline{W}_m\cO_X$ is highly degenerate.

\begin{remark}
The existence of $\rho$ can be also seen as follows (cf.\ \cite[\S 1.1]{OgusVologodsky}): the kernel $VW_{m-1}\cO_X$ of $R^{m-1}\colon W_m\cO_X\to \cO_X$ has a natural divided power structure, and hence so does the kernel $I$ of 
\[ 
	R^{m-1}\colon \overline{W}_m\cO_X\ra \cO_X.
\]
It follows that $f^p=0$ for every $f\in I$, and hence the absolute Frobenius of $\overline{W}_m \cO_X$ factors naturally through $\cO_X$. 
\end{remark}

\begin{prop}
The following diagram is a pushout square
\[ 
	\xymatrix{
		W_m\cO_X \ar[r]^F \ar[d]_{R^{m-1}} & W_m\cO_X \ar[d]^\pi \\
		\cO_X \ar[r]_-\rho & \overline{W}_m\cO_X.
	}
\]
\end{prop}

\begin{proof}
The commutativity of the diagram follows from the formula
\begin{align*}
	F(f_0, \ldots, f_{m-1}) &= (f_0^p, \ldots, f_{m-1}^p) \\
	&= [f_0]^p + VF(f_1, \ldots, f_{m-1}) \\
	&= [f_0]^p + p\cdot (f_1, \ldots, f_{m-1}) \equiv \rho(f_0) \, {\rm mod}\, p.
\end{align*}
Since the vertical arrows are surjective, it remains to check that the induced map 
\[
	F\colon \ker R^{m-1} \ra \ker \pi
\] 
is an isomorphism. This follows from
\[ 
	F(\ker R^{m-1}) = F(V(W_{m-1}\cO_X)) = p\cdot W_{m-1}\cO_X.\qedhere
\]
\end{proof}

\begin{cor}
	Suppose that $X$ is a smooth scheme over $k$. There are natural isomorphisms
	\[
		\overline{W}_m\cO_X\isom \mathcal{F}_m\cO_X
		\quad \text{and} \quad
		B_m \isom B_m\Omega^1_X
	\] 
	fitting inside a commutative diagram
	\[
		\xymatrix{
			0\ar[r] & \cO_X \ar[r]^\rho \ar@{=}[d] & \overline{W}_m\cO_X \ar[d]^{\rotatebox{90}{$\sim$}} \ar[r] & B_m \ar[d]^{\rotatebox{90}{$\sim$}}  \ar[r] & 0 \\
			0 \ar[r] & \cO_X \ar[r] & \mathcal{F}_m\cO_X \ar[r] & B_m\Omega^1_X \ar[r] & 0.
		}
	\]
\end{cor}

We can therefore rephrase the definition of a quasi-$F$-splitting as follows.

\begin{defin}
	Let $X$ be an $\FF_p$-scheme. A \emph{quasi-$F$-splitting of level $m$} on $X$ is an $\cO_X$-linear splitting $\sigma\colon \overline{W}_m\cO_X\to \cO_X$ of the map
	\[
		\rho \colon \cO_X\ra \overline{W}_m\cO_X.
	\]
\end{defin}

\section{Families with vanishing Kodaira--Spencer}
\label{app:zeroks}

\noindent
Let $k$ be a perfect field of characteristic $p>0$ and let $S$ be a smooth $k$-scheme. Let $f\colon X\to S$ be a smooth morphism. Applying $Rf_*$ to the short exact sequence on $X$
\begin{equation} \label{eqn:ks-seq} 
	0\ra T_{X/S} \ra T_{X/k} \ra f^* T_{S/k} \ra 0
\end{equation}
yields an exact sequence on $S$
\begin{equation} \label{eqn:longks} 
	0\ra f_* T_{X/S} \ra f_* T_{X/k} \ra f_* f^* T_{S/k} \xto{\delta} R^1 f_* T_{X/S}.
\end{equation}
The \emph{Kodaira--Spencer map} of $X/S$ is the map
\[ 
	{\rm KS}_{X/S} \colon T_{S/k} \ra R^1 f_* T_{X/S}
\]
obtained as the composition of the adjunction map $T_{S/k}\to f_* f^* T_{S/k}$ and the map $\delta\colon f_* f^* T_{S/k} \to R^1 f_* T_{X/S}$. 

\begin{thm}[{compare \cite[Lemma~3.5]{OgusFGT}}] \label{thm:ks_zero_F_descent}
	Suppose that ${\rm KS}_{X/S} = 0$ and $f_* T_{X/S} = 0$. Then there exists a canonical cartesian diagram
	\begin{equation} \label{eqn:kssq} 
		\xymatrix{
			X \ar[d] \ar[r] & Y \ar[d] \\
			S \ar[r]_{F_{S/k}} & S'.
		}
	\end{equation}
\end{thm}

\begin{proof}
The assumptions ${\rm KS}_{X/S} = 0$ and $f_* T_{X/S} = 0$ combined with the exactness of \eqref{eqn:longks} show that the adjunction map $T_{S/k}\to f_* f^* T_{S/k}$ factors uniquely through a map $u\colon T_{S/k}\to f_* T_{X/k}$. By another adjunction, we obtain a map $v\colon f^* T_{S/k} \to T_{X/k}$ which splits \eqref{eqn:ks-seq}. We check that $v$ defines a $1$-foliation (see \cite[I.1, p.\ 104]{EkedahlCanonical}), i.e.\ that its image is closed under the Lie bracket and $p$-th iterates. The respective obstructions [op.cit. Lemma 1.4, p.\ 105] are maps
\[ 
	\bigwedge^2 f^* T_{S/k} \ra T_{X/S} \quad \text{and} \quad F_{X}^* f^* T_{S/k} \ra T_{X/S}
\]
by adjunction correspond to maps
\[ 
	\bigwedge^2 T_{S/k} \ra f_* T_{X/S} \quad \text{and} \quad F_S^* T_{S/k} \ra f_* T_{X/S}
\]
which both vanish because the target is zero. We define $Y$ to be the quotient by this $1$-foliation.
\end{proof}




\begin{remark}
As the example $S=\mathbf{P}^1_k$ and $X=\mathbf{P}(\cO_S\oplus \cO_{S}(1))\to S$ shows, the assumption that $f_* T_{X/S}=0$ is necessary. Note that in this example $X/S$ still descends along $F_{S/k}$ Zariski-locally on $S$. To produce examples which do not descend even locally, one can take $X/S$ a Brauer--Severi variety whose corresponding class in $H^2_\et(S, \GG_m)$ is not divisible by~$p$.

Of course, if ${\rm KS}_{X/S} = 0$, then the adjunction map $T_{S/k}\to f_* f^* T_{S/k}$ can locally be lifted to a map $u\colon T_{S/k}\to f_* T_{X/k}$ (since $T_{S/k}$ is locally free). However, for the above proof to work, we need the lifting $u$ to be compatible with the restricted Lie algebra structures.
\end{remark}





\bibliographystyle{amsalpha} 
\bibliography{bib.bib}

\providecommand{\bysame}{\leavevmode\hbox to3em{\hrulefill}\thinspace}
\providecommand{\MR}{\relax\ifhmode\unskip\space\fi MR }
\providecommand{\MRhref}[2]{%
  \href{http://www.ams.org/mathscinet-getitem?mr=#1}{#2}
}
\providecommand{\href}[2]{#2}
\begin{thebibliography}{{Yob}17}

\bibitem[AM77]{Artin_Mazur}
Michael Artin and Barry Mazur, \emph{Formal groups arising from algebraic
  varieties}, Ann. Sci. \'Ecole Norm. Sup. (4) \textbf{10} (1977), no.~1,
  87--131.
  \MR{\href{http://www.ams.org/mathscinet-getitem?mr=0457458}{0457458}}

\bibitem[AWZ17]{AchingerWitaszekZdanowicz}
Piotr {Achinger}, Jakub {Witaszek}, and Maciej {Zdanowicz}, \emph{{Liftability
  of the Frobenius morphism and images of toric varieties}},
  \href{https://arxiv.org/abs/1708.03777}{arXiv:1708.03777} (2017).

\bibitem[BG16]{BorgerGuerney}
James {Borger} and Lance {Gurney}, \emph{{Canonical lifts of families of
  elliptic curves}}, \href{https://arxiv.org/abs/1608.05912}{arXiv:1608.05912}
  (2016).

\bibitem[BK86]{Bloch_Kato}
Spencer Bloch and Kazuya Kato, \emph{{$p$}-adic \'etale cohomology}, Inst.
  Hautes \'Etudes Sci. Publ. Math. (1986), no.~63, 107--152.
  \MR{\href{http://www.ams.org/mathscinet-getitem?mr=849653}{849653}}

\bibitem[BO78]{BerthelotOgus}
Pierre Berthelot and Arthur Ogus, \emph{Notes on crystalline cohomology},
  Princeton University Press, Princeton, N.J.; University of Tokyo Press,
  Tokyo, 1978.
  \MR{\href{http://www.ams.org/mathscinet-getitem?mr=0491705}{0491705}}

\bibitem[Bog78]{Bogomolov}
Fedor~A. Bogomolov, \emph{Hamiltonian {K}\"ahlerian manifolds}, Dokl. Akad.
  Nauk SSSR \textbf{243} (1978), no.~5, 1101--1104.
  \MR{\href{http://www.ams.org/mathscinet-getitem?mr=514769}{514769}
  (80c:32024)}

\bibitem[CvS09]{CynkVanStraten}
S{\l}awomir Cynk and Duco van Straten, \emph{Small resolutions and non-liftable
  {C}alabi-{Y}au threefolds}, Manuscripta Math. \textbf{130} (2009), no.~2,
  233--249.
  \MR{\href{http://www.ams.org/mathscinet-getitem?mr=2545516}{2545516}
  (2010k:14074)}

\bibitem[Del81]{DeligneIllusieKatz}
Pierre Deligne, \emph{Cristaux ordinaires et coordonn\'ees canoniques},
  Algebraic surfaces ({O}rsay, 1976--78), Lecture Notes in Math., vol. 868,
  Springer, Berlin-New York, 1981, With the collaboration of L. Illusie, With
  an appendix by Nicholas M. Katz, pp.~80--137.
  \MR{\href{http://www.ams.org/mathscinet-getitem?mr=638599}{638599}}

\bibitem[DI87]{DeligneIllusie}
Pierre Deligne and Luc Illusie, \emph{Rel\`evements modulo {$p^2$} et
  d\'ecomposition du complexe de de {R}ham}, Invent. Math. \textbf{89} (1987),
  no.~2, 247--270.
  \MR{\href{http://www.ams.org/mathscinet-getitem?mr=894379}{894379}}

\bibitem[Eke88]{EkedahlCanonical}
Torsten Ekedahl, \emph{Canonical models of surfaces of general type in positive
  characteristic}, Inst. Hautes \'Etudes Sci. Publ. Math. (1988), no.~67,
  97--144. \MR{\href{http://www.ams.org/mathscinet-getitem?mr=972344}{972344}}

\bibitem[ES15]{EsnaultShiho}
Hélène Esnault and Atsushi Shiho, \emph{Convergent isocrystals on simply
  connected varieties}, 2015.

\bibitem[ESB05]{EkedahlSB}
Torsten Ekedahl and Nicholas~I. Shepherd-Barron, \emph{Tangent lifting of
  deformations in mixed characteristic}, J. Algebra \textbf{291} (2005), no.~1,
  108--128.
  \MR{\href{http://www.ams.org/mathscinet-getitem?mr=2158513}{2158513}
  (2006d:13015)}

\bibitem[Fin10]{Finotti}
Lu\'\i s R.~A. Finotti, \emph{Lifting the {$j$}-invariant: questions of {M}azur
  and {T}ate}, J. Number Theory \textbf{130} (2010), no.~3, 620--638.
  \MR{\href{http://www.ams.org/mathscinet-getitem?mr=2584845}{2584845}}

\bibitem[GS06]{GrossSiebert}
Mark Gross and Bernd Siebert, \emph{Mirror symmetry via logarithmic
  degeneration data. {I}}, J. Differential Geom. \textbf{72} (2006), no.~2,
  169--338.
  \MR{\href{http://www.ams.org/mathscinet-getitem?mr=2213573}{2213573}}

\bibitem[Har77]{Hartshorne}
Robin Hartshorne, \emph{Algebraic geometry}, Springer-Verlag, New York, 1977,
  Graduate Texts in Mathematics, No. 52.
  \MR{\href{http://www.ams.org/mathscinet-getitem?mr=0463157}{0463157} (57
  \#3116)}

\bibitem[Hir99]{Hirokado}
Masayuki Hirokado, \emph{A non-liftable {C}alabi-{Y}au threefold in
  characteristic {$3$}}, Tohoku Math. J. (2) \textbf{51} (1999), no.~4,
  479--487.
  \MR{\href{http://www.ams.org/mathscinet-getitem?mr=1725623}{1725623}
  (2000m:14044)}

\bibitem[HM03]{HesselholtMadsen}
Lars Hesselholt and Ib~Madsen, \emph{On the {$K$}-theory of local fields}, Ann.
  of Math. (2) \textbf{158} (2003), no.~1, 1--113.
  \MR{\href{http://www.ams.org/mathscinet-getitem?mr=1998478}{1998478}}

\bibitem[Ill79]{Illusie_deRhamWitt}
Luc Illusie, \emph{Complexe de de\thinspace {R}ham-{W}itt et cohomologie
  cristalline}, Ann. Sci. \'Ecole Norm. Sup. (4) \textbf{12} (1979), no.~4,
  501--661. \MR{\href{http://www.ams.org/mathscinet-getitem?mr=565469}{565469}}

\bibitem[Ill96]{IllusieFrobenius}
\bysame, \emph{Frobenius et d\'eg\'en\'erescence de {H}odge}, Introduction \`a\
  la th\'eorie de {H}odge, Panor. Synth\`eses, vol.~3, Soc. Math. France,
  Paris, 1996, pp.~113--168.
  \MR{\href{http://www.ams.org/mathscinet-getitem?mr=1409820}{1409820}}

\bibitem[IR83]{IllusieRaynaud}
Luc Illusie and Michel Raynaud, \emph{Les suites spectrales associ\'ees au
  complexe de de {R}ham-{W}itt}, Inst. Hautes \'Etudes Sci. Publ. Math. (1983),
  no.~57, 73--212.
  \MR{\href{http://www.ams.org/mathscinet-getitem?mr=699058}{699058}}

\bibitem[Kat72]{KatzAlgSoln}
Nicholas~M. Katz, \emph{Algebraic solutions of differential equations
  ({$p$}-curvature and the {H}odge filtration)}, Invent. Math. \textbf{18}
  (1972), 1--118.
  \MR{\href{http://www.ams.org/mathscinet-getitem?mr=0337959}{0337959}}

\bibitem[Kat73]{KatzDwork}
\bysame, \emph{Travaux de {D}work}, 167--200. Lecture Notes in Math., Vol. 317.
  \MR{\href{http://www.ams.org/mathscinet-getitem?mr=0498577}{0498577}}

\bibitem[Kat79]{KatzSlopeFil}
\bysame, \emph{Slope filtration of {$F$}-crystals}, Journ\'ees de
  {G}\'eom\'etrie {A}lg\'ebrique de {R}ennes ({R}ennes, 1978), {V}ol. {I},
  Ast\'erisque, vol.~63, Soc. Math. France, Paris, 1979, pp.~113--163.
  \MR{\href{http://www.ams.org/mathscinet-getitem?mr=563463}{563463}}

\bibitem[Kat81]{KatzSerreTate}
\bysame, \emph{Serre-{T}ate local moduli}, Algebraic surfaces ({O}rsay,
  1976--78), Lecture Notes in Math., vol. 868, Springer, Berlin-New York, 1981,
  pp.~138--202.
  \MR{\href{http://www.ams.org/mathscinet-getitem?mr=638600}{638600}}

\bibitem[Kaw92]{Kawamata}
Yujiro Kawamata, \emph{Unobstructed deformations. {A} remark on a paper of {Z}.
  {R}an: ``{D}eformations of manifolds with torsion or negative canonical
  bundle'' [{J}. {A}lgebraic {G}eom.\ {\bf 1} (1992), no.\ 2, 279--291;
  {MR}1144440 (93e:14015)]}, J. Algebraic Geom. \textbf{1} (1992), no.~2,
  183--190.
  \MR{\href{http://www.ams.org/mathscinet-getitem?mr=1144434}{1144434}
  (93e:14016)}

\bibitem[Kaw97]{KawamataErratum}
\bysame, \emph{Erratum on: ``{U}nobstructed deformations. {A} remark on a paper
  of {Z}. {R}an: `{D}eformations of manifolds with torsion or negative
  canonical bundle'\,'' [{J}. {A}lgebraic {G}eom.\ {\bf 1} (1992), no.\ 2,
  183--190; {MR}1144434 (93e:14016)]}, J. Algebraic Geom. \textbf{6} (1997),
  no.~4, 803--804.
  \MR{\href{http://www.ams.org/mathscinet-getitem?mr=1487238}{1487238}
  (98k:14009)}

\bibitem[LS77]{Lange_Stuhler}
Herbert Lange and Ulrich St\"uhler, \emph{Vektorb\"undel auf {K}urven und
  {D}arstellungen der algebraischen {F}undamentalgruppe}, Math. Z. \textbf{156}
  (1977), no.~1, 73--83.
  \MR{\href{http://www.ams.org/mathscinet-getitem?mr=0472827}{0472827}}

\bibitem[LS13]{LiuShen}
Kefeng {Liu} and Yang {Shen}, \emph{{Hodge metric completion of the moduli
  space of Calabi-Yau manifolds}}, ArXiv e-prints (2013), arXiv:1305.0231.

\bibitem[Maz73]{Mazur}
Barry Mazur, \emph{Frobenius and the {H}odge filtration (estimates)}, Ann. of
  Math. (2) \textbf{98} (1973), 58--95.
  \MR{\href{http://www.ams.org/mathscinet-getitem?mr=0321932}{0321932}}

\bibitem[MB85]{MoretBailly}
Laurent Moret-Bailly, \emph{Pinceaux de vari\'et\'es ab\'eliennes},
  Ast\'erisque (1985), no.~129, 266.
  \MR{\href{http://www.ams.org/mathscinet-getitem?mr=797982}{797982}}

\bibitem[Moc96]{Mochizuki}
Shinichi Mochizuki, \emph{A theory of ordinary {$p$}-adic curves}, Publ. Res.
  Inst. Math. Sci. \textbf{32} (1996), no.~6, 957--1152.
  \MR{\href{http://www.ams.org/mathscinet-getitem?mr=1437328}{1437328}}

\bibitem[MR85]{MehtaRamanathan}
Vikram~B. Mehta and Annamalai Ramanathan, \emph{Frobenius splitting and
  cohomology vanishing for {S}chubert varieties}, Ann. of Math. (2)
  \textbf{122} (1985), no.~1, 27--40.
  \MR{\href{http://www.ams.org/mathscinet-getitem?mr=799251}{799251}}

\bibitem[MS87]{MehtaSrinivas}
Vikram~B. Mehta and Vasudevan Srinivas, \emph{Varieties in positive
  characteristic with trivial tangent bundle}, Compositio Math. \textbf{64}
  (1987), no.~2, 191--212, With an appendix by Srinivas and M. V. Nori.
  \MR{\href{http://www.ams.org/mathscinet-getitem?mr=916481}{916481}}

\bibitem[NO85]{NygaardOgus}
Niels~O. Nygaard and Arthur Ogus, \emph{Tate's conjecture for {$K3$} surfaces
  of finite height}, Ann. of Math. (2) \textbf{122} (1985), no.~3, 461--507.
  \MR{\href{http://www.ams.org/mathscinet-getitem?mr=819555}{819555}}

\bibitem[Nyg83]{Nygaard}
Niels~O. Nygaard, \emph{The {T}ate conjecture for ordinary {$K3$}\ surfaces
  over finite fields}, Invent. Math. \textbf{74} (1983), no.~2, 213--237.
  \MR{\href{http://www.ams.org/mathscinet-getitem?mr=723215}{723215}}

\bibitem[Ogu78a]{OgusFGT}
Arthur Ogus, \emph{{$F$}-crystals and {G}riffiths transversality}, Proceedings
  of the {I}nternational {S}ymposium on {A}lgebraic {G}eometry ({K}yoto
  {U}niv., {K}yoto, 1977), Kinokuniya Book Store, Tokyo, 1978, pp.~15--44.
  \MR{\href{http://www.ams.org/mathscinet-getitem?mr=578852}{578852}}

\bibitem[Ogu78b]{OgusGT}
\bysame, \emph{Griffiths transversality in crystalline cohomology}, Ann. of
  Math. (2) \textbf{108} (1978), no.~2, 395--419.
  \MR{\href{http://www.ams.org/mathscinet-getitem?mr=506993}{506993}}

\bibitem[Ols07]{Olsson}
Martin~C. Olsson, \emph{Crystalline cohomology of algebraic stacks and
  {H}yodo-{K}ato cohomology}, Ast\'erisque (2007), no.~316, 412 pp. (2008).
  \MR{\href{http://www.ams.org/mathscinet-getitem?mr=2451400}{2451400}}

\bibitem[OV07]{OgusVologodsky}
Arthur Ogus and Vadim Vologodsky, \emph{Nonabelian {H}odge theory in
  characteristic {$p$}}, Publ. Math. Inst. Hautes \'Etudes Sci. (2007),
  no.~106, 1--138.
  \MR{\href{http://www.ams.org/mathscinet-getitem?mr=2373230}{2373230}}

\bibitem[Ran92]{Ran}
Ziv Ran, \emph{Deformations of manifolds with torsion or negative canonical
  bundle}, J. Algebraic Geom. \textbf{1} (1992), no.~2, 279--291.
  \MR{\href{http://www.ams.org/mathscinet-getitem?mr=1144440}{1144440}
  (93e:14015)}

\bibitem[Sch03]{SchroerT1}
Stefan Schr{\"o}er, \emph{The {$T^1$}-lifting theorem in positive
  characteristic}, J. Algebraic Geom. \textbf{12} (2003), no.~4, 699--714.
  \MR{\href{http://www.ams.org/mathscinet-getitem?mr=1993761}{1993761}
  (2004i:14011)}

\bibitem[Sch04]{SchroerExamples}
\bysame, \emph{Some {C}alabi-{Y}au threefolds with obstructed deformations over
  the {W}itt vectors}, Compos. Math. \textbf{140} (2004), no.~6, 1579--1592.
  \MR{\href{http://www.ams.org/mathscinet-getitem?mr=2098403}{2098403}
  (2005i:14051)}

\bibitem[Ser06]{Sernesi}
Edoardo Sernesi, \emph{Deformations of algebraic schemes}, Grundlehren der
  Mathematischen Wissenschaften [Fundamental Principles of Mathematical
  Sciences], vol. 334, Springer-Verlag, Berlin, 2006.
  \MR{\href{http://www.ams.org/mathscinet-getitem?mr=2247603}{2247603}}

\bibitem[Tia87]{Tian}
Gang Tian, \emph{Smoothness of the universal deformation space of compact
  {C}alabi-{Y}au manifolds and its {P}etersson-{W}eil metric}, Mathematical
  aspects of string theory ({S}an {D}iego, {C}alif., 1986), Adv. Ser. Math.
  Phys., vol.~1, World Sci. Publishing, Singapore, 1987, pp.~629--646.
  \MR{\href{http://www.ams.org/mathscinet-getitem?mr=915841}{915841}}

\bibitem[Tod89]{Todorov}
Andrey~N. Todorov, \emph{The {W}eil-{P}etersson geometry of the moduli space of
  {${\rm SU}(n\geq 3)$} ({C}alabi-{Y}au) manifolds. {I}}, Comm. Math. Phys.
  \textbf{126} (1989), no.~2, 325--346.
  \MR{\href{http://www.ams.org/mathscinet-getitem?mr=1027500}{1027500}
  (91f:32022)}

\bibitem[vdGK03]{GeerKatsura}
Gerard van~der Geer and Toshiyuki Katsura, \emph{On the height of
  {C}alabi-{Y}au varieties in positive characteristic}, Doc. Math. \textbf{8}
  (2003), 97--113.
  \MR{\href{http://www.ams.org/mathscinet-getitem?mr=2029163}{2029163}}

\bibitem[War14]{WardThesis}
Matthew Ward, \emph{Arithmetic {P}roperties of the {D}erived {C}ategory for
  {C}alabi-{Y}au {V}arieties}, ProQuest LLC, Ann Arbor, MI, 2014, Thesis
  (Ph.D.)--University of Washington.
  \MR{\href{http://www.ams.org/mathscinet-getitem?mr=3271876}{3271876}}

\bibitem[{Yob}17]{Yobuko}
Fuetaro {Yobuko}, \emph{{Quasi-Frobenius-splitting and lifting of Calabi-Yau
  varieties in characteristic $p$}},
  \href{https://arxiv.org/abs/1704.05604}{arXiv:1704.05604} (2017).

\bibitem[Zda18]{Zdanowicz}
Maciej Zdanowicz, \emph{Liftability of singularities and their {F}robenius
  morphism modulo $p^2$}, International Mathematics Research Notices
  \textbf{2018} (2018), no.~14, 4513--4577,
  \url{http://dx.doi.org/10.1093/imrn/rnw297}.

\end{thebibliography}


\begin{thebibliography}{plain}

	
\bibitem[SGA 5]{SGA5}
\emph{Cohomologie {$\ell$}-adique et fonctions {$L$}}, Lecture Notes in 
	Mathematics, Vol. 589, Springer-Verlag, Berlin, 1977, S{\'e}minaire de 
	G{\'e}ometrie Alg{\'e}brique du Bois-Marie 1965--1966 (SGA 5), Edit{\'e} par Luc Illusie. 
	\MR{\href{http://www.ams.org/mathscinet-getitem?mr=0491704}{0491704} (58 \#10907)}

\end{thebibliography}

\renewcommand{\refname}{\rule{2cm}{0.4pt}}
\renewcommand{\addcontentsline}[3]{}

\end{document}